\documentclass[11pt,a4paper]{article}
\usepackage[utf8]{inputenc}
\usepackage{amsmath,amssymb,amsfonts,amsthm}
\usepackage{hyperref}
\usepackage{geometry}
\numberwithin{equation}{section}
\usepackage{mathtools}
\usepackage{mathrsfs}
\usepackage{xcolor}

\theoremstyle{plain}
\newtheorem{thm}{Theorem}[section]
\newtheorem{prop}[thm]{Proposition}
\newtheorem{lem}[thm]{Lemma}
\newtheorem{cor}[thm]{Corollary}
\theoremstyle{definition}
\newtheorem{defn}[thm]{Definition}
\theoremstyle{remark}
\newtheorem{rmk}[thm]{Remark}
\newtheorem{ex}[thm]{Example}

\title{\textbf{Local Bifurcations from Single-Mode Traveling Waves
in the Filamentation Equation}}
\author{Yuri Cacchi\`o\\
\scriptsize{Faculty of Mathematics, University of Vienna, Oskar-Morgenstern-Platz 1, 1090 Vienna, Austria.}\\ \href{mailto:yuri.cacchio@univie.ac.at}{\scriptsize{yuri.cacchio@univie.ac.at}}}

\begin{document}

\maketitle

\begin{abstract}
We study local bifurcations from the single-mode traveling wave family
of the filamentation equation on the torus, within the real
positive-frequency Sobolev space $X^s$. A Fourier null structure
removes the apparent derivative loss and makes the traveling wave map
real analytic for $s>3/2$. Its linearization splits into finitely many
coupled two-mode blocks, a carrier block, and a diagonal
high-frequency tail. For negative temporal frequency, we determine the
critical values generated by both the finite coupled blocks and the
tail. At every non resonant finite-block critical value,
Lyapunov-Schmidt reduction yields a locally unique real-analytic
branch, with vanishing linear correction to the bifurcation parameter.
Moreover, each finite tail critical value produces a simple local
branch. These values accumulate at a parameter where the
linearization ceases to be Fredholm. The accumulation point is
nevertheless a bifurcation point in the usual topological sense. For
$\sigma=1$, a symmetry in the first Fourier mode produces an
exact two-mode vertical branch, including the exceptional
non-Fredholm case $k=2$.

\vspace{0.5cm}

\end{abstract}

\tableofcontents

\section{Introduction}
 
\subsection{Filamentation on two-dimensional vorticity interfaces}
\label{sub:filamentation}
 
For the two-dimensional incompressible Euler equations, vorticity is
materially conserved. Therefore, an initially piecewise-constant
vorticity distribution remains piecewise constant, and its evolution
is entirely determined by the motion of the interfaces, or contours,
separating regions of constant vorticity. This is the formulation known as \emph{contour dynamics}
\cite{Dritschel1989, ZabuskyHughesRoberts1979}, and it takes the same form on the
sphere as in the plane \cite{Dritschel1988sphere}. For a single interface
$\mathcal{C}$ across which the vorticity jumps by a constant amount $\chi$,
each point $\mathbf{x}\in\mathcal{C}$ moves according to
\begin{equation}
  \frac{d\mathbf{x}}{dt}
  \;=\;-\,\frac{\chi}{4\pi}\oint_{\mathcal{C}}
  \log|\mathbf{x}'-\mathbf{x}|^{2}\,d\mathbf{x}' .
  \label{eq:contour-dynamics}
\end{equation}
Interfaces of this kind retain their regularity for all time
\cite{BertozziConstantin1993, Chemin1993}, and yet they are observed to lengthen
and to fold without bound.
 
The mechanism responsible for this growth in complexity was identified by
Dritschel \cite{Dritschel1988filamentation}. A small disturbance riding on a
circular or on a straight equilibrium interface steepens, overturns and sheds a
thin filament. Once the first filament has formed, the process repeats, and the
interface never returns to a simple shape. For the history of the problem,
which goes back to Kelvin \cite{Thomson1880}, we refer to Craik
\cite{Craik2012}.
 
The onset of filamentation, that is the wave-steepening stage which precedes the
overturning, admits an asymptotic description. In the linearized theory, every perturbation of an equilibrium interface oscillates with the same frequency $\chi/2$, independently of its spatial structure. The linearized dynamics is
therefore purely oscillatory and carries no dispersion, so that no steepening
can take place at that order, and the first correction appears at cubic
order. Writing the interface displacement in the form
\begin{equation*}
  \eta(\theta,t)\;=\;A(\theta,t)\,e^{\,i\chi t/2}+\mathrm{c.c.},
\end{equation*}
with $A$ a slowly varying amplitude, expanding \eqref{eq:contour-dynamics} to
third order in the wave amplitude and removing the secular terms, Dritschel
\cite{Dritschel1988filamentation} obtained a closed cubic evolution equation for
$A$. Dritschel, Constantin and Germain \cite{DritschelConstantinGermain2025}
observed that, in a suitably rescaled slow time, this amplitude equation carries
no free parameter. The same equation governs a circular vortex patch on
the sphere, at any latitude, and a circular patch in the plane, while a straight
periodic interface in the plane obeys the same equation with one cubic term
removed. The two geometries are thus collected in a single equation depending on
a parameter $\sigma\in\{0,1\}$.
 
\subsection{Filamentation equation}
\label{sub:the-equation}
 
Let $\mathbb{T}=\mathbb{R}/(2\pi\mathbb{Z})$. For
$u(x)=\sum_{m\in\mathbb{Z}}\widehat{u}_{m}e^{imx}$ we denote by $\mathbb P_{+}$ the
projector onto the strictly positive frequencies,
$\mathbb P_{+}e^{imx}=\mathbf{1}_{[1,\infty)}(m)\,e^{imx}$, and we set
\begin{equation*}
  \widehat{\Lambda u}_{m}=|m|\,\widehat{u}_{m},
  \qquad
  \widehat{\Lambda^{-1}u}_{m}=
  \begin{cases}
    |m|^{-1}\widehat{u}_{m}, & m\neq 0,\\[2pt]
    0, & m=0,
  \end{cases}
\end{equation*}
so that $\Lambda=|\partial_{x}|$.
 
Constantin, Dritschel and Germain \cite{ConstantinDritschelGermain2025} noticed
that the amplitude equation described above can be written, for complex valued
functions $u=u(t,x)$ with positive spectrum $\mathbb P_{+}u=u$, in the compact form
\begin{equation}
  \partial_{t}u \;=\; \partial_{x}\,\mathcal{C}_{\sigma}[u],
  \label{eq:filamentation}
\end{equation}
where the cubic term is
\begin{equation}
  \mathcal{C}_{\sigma}[u](x)
  \;=\;\mathbb P_{+}\left[\,
  \frac{1}{4\pi}\int_{0}^{2\pi}
  \frac{|u(x)-u(y)|^{2}\bigl(u(x)-u(y)\bigr)}{1-\cos(x-y)}\,dy
  \;-\;\sigma\,|u(x)|^{2}u(x)\right].
  \label{eq:cubic}
\end{equation}
Equation \eqref{eq:filamentation} is the \emph{filamentation equation}. The
value $\sigma=0$ corresponds to a periodic vorticity interface in the
plane, and $\sigma=1$ to a zonal vorticity interface on the sphere.
 
Two equivalent expressions for the cubic term are used throughout. The first is
local in $\Lambda$,
\begin{equation}
  \mathcal{C}_{\sigma}[u]
  \;=\;\mathbb P_{+}\Bigl[\,|u|^{2}\Lambda u-u\,\Lambda|u|^{2}-\sigma|u|^{2}u\Bigr],
  \label{eq:cubic-Lambda}
\end{equation}
and the second is spectral: writing $u=\sum_{m\geq 1}a_{m}e^{imx}$,
\begin{equation}
  \mathcal{C}_{\sigma}[u]
  \;=\;\sum_{\substack{k,\ell,m,p\,\geq\,1\\ p=k+\ell-m}}
  \bigl(\min(k,\ell,m,p)-\sigma\bigr)\,a_{k}a_{\ell}\overline{a_{m}}\;e^{ipx}.
  \label{eq:cubic-Fourier}
\end{equation}
The symmetrized symbol in \eqref{eq:cubic-Fourier} comes from the identity
\begin{equation*}
  \tfrac{1}{2}\bigl(k+\ell-|k-p|-|\ell-p|\bigr)=\min(k,\ell,m,p),
  \qquad k+\ell=m+p,\quad k,\ell,m,p\geq 1,
\end{equation*}
already noted by Biello and Hunter \cite{BielloHunter2010}. This identity is
the analytic heart of the problem. Each of the two terms $|u|^{2}\Lambda u$ and
$u\Lambda|u|^{2}$ in \eqref{eq:cubic-Lambda} loses one derivative, but their
difference does not (see Lemma \ref{lem:null} below). Indeed, in \eqref{eq:cubic-Fourier}, the derivative is carried by the \emph{lowest} of the four frequencies involved.
 
\subsection{Known results}
\label{sub:known-results}
The filamentation equation is Hamiltonian. Its energy is
\begin{equation}\label{eq:energy}
  E_{\sigma}(u)=\frac{1}{4\pi^{2}}\iint_{0}^{2\pi}
  \frac{|u(x)-u(y)|^{4}}{1-\cos(x-y)}\,dx\,dy
  -\frac{2\sigma}{\pi}\int_{0}^{2\pi}|u(x)|^{4}\,dx.
\end{equation}
For the symplectic form
\begin{equation}\label{eq:symp_form}
      \omega(u,v)
  :=
  \operatorname{Re}\int_0^{2\pi}
  (\partial_x^{-1}u)\,\overline v\,dx,
\end{equation}
Appendix~A of \cite{ConstantinDritschelGermain2025} gives
\[
  \nabla_{\omega}E_\sigma(u)
  =
  \frac{8}{\pi}\,\partial_x\mathcal C_\sigma[u].
\]
Therefore, \eqref{eq:filamentation} is the Hamiltonian flow generated by
$\frac{\pi}{8}E_\sigma$ for the symplectic form \eqref{eq:symp_form}.
Equivalently, it is the flow generated by $E_\sigma$ after a constant
rescaling of time.  

The first term in \eqref{eq:energy} is comparable to the fourth power of the Gagliardo seminorm $[u]_{W^{1/4,4}(\mathbb T)}$.
Since positive-frequency functions have zero mean, this seminorm is equivalent to the full $W^{1/4,4}$ norm.  Thus $W^{1/4,4}(\mathbb T)$ is the natural energy space \cite{Triebel1992}.

Besides $E_{\sigma}$, the momentum $P(u)=\int|u|^{2}$ and the mass $M(u)=\int|\partial_{x}^{-1/2}u|^{2}$ are
conserved, and for $\sigma=1$ so is the first Fourier coefficient $a_{1}$. The
equation is invariant under space translations $u\mapsto u(\cdot+x_{0})$, phase
rotations $u\mapsto e^{i\theta}u$, the scaling
$u(t,x)\mapsto \mu\,u(\mu^{2}t,x)$, and the space-time reversal
$u(t,x)\mapsto\overline{u}(-t,-x)$.
 
For $\sigma=0$, equation \eqref{eq:filamentation} is also the cubic normal form
of the Burgers-Hilbert equation
$\partial_{t}f+f\partial_{x}f=\mathcal{H}f$, which was written down by Marsden
and Weinstein \cite{MarsdenWeinstein1983} and proposed by Biello and Hunter
\cite{BielloHunter2010} as a universal model for Hamiltonian waves of constant
linear frequency, vorticity interfaces being the guiding example. The
Burgers-Hilbert equation has been much studied over the last decade; see
\cite{BressanNguyen2014,CastroCordobaZheng2023,
DahneGomezSerrano2023,HunterIfrimTataruWong2015,Yang2021} and the review \cite{Hunter2018}. In the regime of small
amplitudes, however, cubic terms cannot be discarded, and
\eqref{eq:filamentation} rather than the Burgers-Hilbert equation is the
correct asymptotic model \cite{ConstantinDritschelGermain2025}.
 
For equation \eqref{eq:filamentation}, local \(H^s\) solutions exist for \(s>3/2\),
while uniqueness is known for \(s>5/2\)
\cite{ConstantinDritschelGermain2025,HunterMorenoVasquezShuZhang2022}.
Global weak solutions in \(L^\infty_tW^{1/4,4}\) were constructed in
\cite{ConstantinDritschelGermain2025}, using the coercivity of the
Hamiltonian. 
The traveling waves of \eqref{eq:filamentation} are known at the level of
existence. Since $\mathcal{C}_{\sigma}[e^{ikx}]=(k-\sigma)e^{ikx}$, the wave supported on a single Fourier mode
\begin{equation}
  \Psi_{k}(t,x)=e^{\,i[kx+k(k-\sigma)t]},\qquad k\in\mathbb{N},
  \label{eq:explicit-tw}
\end{equation}
solves \eqref{eq:filamentation} for every $k$. In
\cite{ConstantinDritschelGermain2025} the wave $\Psi_{1}$ was shown to be
orbitally stable for $\sigma=0$, and $\Psi_{1}$ and $\Psi_{2}$ for $\sigma=1$.
Moreover, minimizing $E_{\sigma}$ under fixed mass and momentum produces
traveling waves in the energy space $W^{1/4,4}$.

For \(\sigma=1\), formula \eqref{eq:cubic-Fourier} shows that the nonlinearity is independent of the first Fourier coefficient
\[
  \mathcal C_1[v+Ae^{ix}]
  =
  \mathcal C_1[v]
  \qquad
  \text{for every }A\in\mathbb C
  \text{ and every }v\text{ such that }\mathbb P_+v=v.
\]
Therefore, evolution \eqref{eq:filamentation} has the symmetry
\[
  u(t,x)\longmapsto u(t,x)+Ae^{ix},
  \qquad A\in\mathbb C.
\]
In the Hamiltonian formulation, this symmetry corresponds via
Noether's theorem to the conservation of the two real components
of \(a_1\) recorded in
\cite{ConstantinDritschelGermain2025}.  Taking $v=Be^{ikx}$ with
$k\ge2$ gives $\mathcal C_{1}[Ae^{ix}+Be^{ikx}]=(k-1)|B|^{2}B\,e^{ikx}$,
so that
\[
  u(t,x)=Ae^{ix}+Be^{\,i[kx+k(k-1)|B|^{2}t]}
\]
is an exact solution. 
This family of solutions is reminiscent of the Rossby-Haurwitz waves of
planetary flows \cite{Haurwitz1940,Vallis2006}.
For the Euler equation on a rotating
sphere, the bifurcation of non-zonal flows of Rossby-Haurwitz type has been
studied in \cite{yuri1,yuri2,ConstantinGermain2022}.

More specifically, the numerical simulations in
\cite[Section~3.4]{DritschelConstantinGermain2025} consider
representative initial data supported on the second and third Fourier
modes. They suggest the formation of a finite-time wave-slope
singularity, interpreted there as the onset of filamentation. These
computations provide dynamical motivation for analyzing the local set
of traveling waves near the single-mode family.
We carry out this local analysis in the subspace $X^s$, consisting of positive-frequency profiles with real Fourier coefficients, or equivalently satisfying $\Phi(-x)=\overline{\Phi(x)}$ (see Section \ref{sec:functional}).

This restriction removes the neutral direction generated by phase rotations and spatial translations, providing a natural setting for the local bifurcation problem. Within $X^s$, we identify the critical parameters and construct the corresponding local real-analytic branches bifurcating from the explicit waves.

\subsection{Bifurcation problem}
\label{sub:bifurcation-problem}
 
Traveling waves of \eqref{eq:filamentation} are the solutions of the form
\begin{equation}\label{eq:tw_sol}
  u(t,x)=\Phi(x-ct)\,e^{i\omega t},
\end{equation}
where the profile $\Phi$ has positive spectrum, $c\in\mathbb{R}$ is the wave
speed and $\omega\in\mathbb{R}$ the temporal frequency. The cubic term
$\mathcal{C}_{\sigma}$ contains two factors $u$ and one factor $\overline{u}$
and commutes with translations, so that
$\mathcal{C}_{\sigma}\bigl[\Phi(\cdot-ct)e^{i\omega t}\bigr]
=e^{i\omega t}\,\mathcal{C}_{\sigma}[\Phi](\cdot-ct)$, and the ansatz gives
\begin{equation*}
  -c\,\partial_{x}\Phi+i\omega\,\Phi=\partial_{x}\mathcal{C}_{\sigma}[\Phi].
\end{equation*}
For positive frequency functions, we have
$\partial_{x}^{-1}=-i\Lambda^{-1}$. Applying $\partial_{x}^{-1}$ we obtain the
traveling wave equation
\begin{equation}\label{eq:TW}
  -c\,\Phi+\omega\,\Lambda^{-1}\Phi=\mathcal{C}_{\sigma}[\Phi],
\end{equation}
where, by \eqref{eq:cubic-Lambda},
\begin{equation}
  \mathcal{C}_{\sigma}[\Phi]
  =\mathbb P_{+}\Bigl[\,|\Phi|^{2}\Lambda\Phi-\Phi\,\Lambda|\Phi|^{2}
  -\sigma|\Phi|^{2}\Phi\Bigr].
  \label{eq:cubic-TW}
\end{equation}
We fix $\omega\in\mathbb{R}$ and $\sigma\in\{0,1\}$, and we choose a
wavenumber $k\geq 1$ with $k>\sigma$. Thus, we consider the wave
\begin{equation}
  \Phi_{0}(x)=A\,e^{ikx},\qquad A>0.
  \label{eq:mono}
\end{equation}
Since $|\Phi_0|^2=A^2$, $\Lambda|\Phi_0|^2=0$, and
$\mathbb P_+$ acts as the identity on $e^{ikx}$ because $k>0$, we have
\begin{equation*}
  \mathcal{C}_{\sigma}[\Phi_{0}]
  =\mathbb P_{+}\bigl[A^{2}\Lambda\Phi_{0}-0-\sigma A^{2}\Phi_{0}\bigr]
  =A^{3}(k-\sigma)\,e^{ikx},
\end{equation*}
while $\Lambda^{-1}\Phi_{0}=k^{-1}Ae^{ikx}$. Equation \eqref{eq:TW} evaluated at
$\Phi_{0}$ reads
\begin{equation*}
  -cA+\omega\frac{A}{k}=A^{3}(k-\sigma),
\end{equation*}
and, after division by $A\neq 0$, we derive
\begin{equation}
  -c+\frac{\omega}{k}=A^{2}(k-\sigma).
  \label{eq:dispersion}
\end{equation}
We take
\begin{equation*}
  \lambda=A^{2}>0
\end{equation*}
as bifurcation parameter and eliminate the speed
through \eqref{eq:dispersion}, that is
\begin{equation*}
  c(\lambda)=\frac{\omega}{k}-\lambda(k-\sigma).
\end{equation*}
Then, the traveling wave equation \eqref{eq:TW} becomes $\mathcal F(\Phi,\lambda)=0$, with
\begin{equation}
  \mathcal F(\Phi,\lambda):=
  -\Bigl(\frac{\omega}{k}-\lambda(k-\sigma)\Bigr)\Phi
  +\omega\Lambda^{-1}\Phi-\mathcal C_\sigma[\Phi].
  \label{eq:F}
\end{equation}
For every \(\lambda>0\), the profile
\[
\Phi_0(\lambda)=\sqrt{\lambda}\,e^{ikx}
\]
satisfies
\begin{equation}
\mathcal F\bigl(\Phi_{0}(\lambda),\lambda\bigr)=0.
  \label{eq:trivial-branch}
\end{equation}
The curve
\[
\lambda\longmapsto
\bigl(\Phi_0(\lambda),\lambda\bigr)
\]
will be referred to as the trivial branch.

Three features of this formulation deserve a comment.
\begin{itemize}
\item \emph{The redundancy of \(c\) and \(\omega\).} For
\(\Phi_0(\lambda)=\sqrt{\lambda}\,e^{ikx}\), the associated solution
\eqref{eq:tw_sol} depends on \(c\) and \(\omega\) only through the
combination \(\omega-kc\). Consequently, replacing \((c,\omega)\) by
\((c+\delta,\omega+k\delta)\) leaves the solution unchanged. Along the
trivial branch, the dispersion relation gives
\[
  \omega-kc(\lambda)=\lambda k(k-\sigma),
\]
so the corresponding solution
\[
  u(t,x)
  =
  \sqrt{\lambda}\,
  e^{\,i[kx+\lambda k(k-\sigma)t]}
\]
is independent of \(\omega\).

However, if $\Phi$ contains a second Fourier mode $q\neq k$, this
non-uniqueness disappears.  Indeed, the same transformation gives for the temporal frequency of the q-th mode
\[
  \omega-qc+(k-q)\delta.
\]
Since $q\neq k$, the resulting solution is unchanged only when
$\delta=0$.  Thus, for profiles containing more than one Fourier mode,
$c$ and $\omega$ are determined separately.  We therefore fix
$\omega$ in the bifurcation problem.  The critical values of $\lambda$
found below are positive only when $\omega<0$, or equivalently, by
\eqref{eq:dispersion}, when
\[
  c(\lambda)<-\lambda(k-\sigma).
\]
See discussion in Section \ref{sec:thresholds} for the further details.
\item \emph{The symmetries.}
Equation \eqref{eq:TW} is invariant under phase rotations and spatial
translations.  At the trivial profile $\Phi_{0}$, the infinitesimal
generators of these two symmetries are
\begin{equation}\label{eq_gen_symm}
     \left.\frac{d}{d\theta}
  \bigl(e^{i\theta}\Phi_{0}\bigr)\right|_{\theta=0}
  =i\Phi_{0},
  \qquad
  \left.\frac{d}{dx_{0}}
  \Phi_{0}(\,\cdot+x_{0})\right|_{x_{0}=0}
  =\partial_{x}\Phi_{0}
  =ik\Phi_{0}.
\end{equation}
They are therefore proportional and generate the same neutral
direction, namely $i\Phi_{0}$.  Rather than quotienting by the symmetry
group, we work in the closed real subspace $X^{s}$ introduced in
Section~\ref{sec:functional}.  Since $i\Phi_{0}$ has a purely imaginary
Fourier coefficient, it does not belong to $X^{s}$.  
\item 
\emph{The derivative loss.} The map $\mathcal F$ contains the term $|\Phi|^{2}\Lambda\Phi$
and, read term by term, loses one derivative: no implicit function theorem can be
applied directly. The null structure recalled after
\eqref{eq:cubic-Fourier} removes this loss, and we show in Section
\ref{sec:functional} that $\mathcal F$ is in fact real analytic on the considered functional space. The bifurcation problem then
falls within the classical analytic Lyapunov-Schmidt framework
\cite{BuffoniToland2003,CrandallRabinowitz1971,Kielhofer2012}, as used for
instance in \cite{ConstantinGermain2022,ConstantinStrauss2004} for other
problems in fluid mechanics.
\end{itemize}

\subsection{Results and plan of the paper}
\label{sub:results}
The proof follows the scheme of bifurcation from a simple
eigenvalue \cite{CrandallRabinowitz1971}, in the real-analytic setting
of \cite{BuffoniToland2003}.  One writes the traveling wave equation as
$\mathcal F(\Phi,\lambda)=0$ on $X^{s}\times(0,\infty)$, with the
explicit trivial branch $\Phi_{0}(\lambda)$, and checks that
$\mathcal F$ is real analytic. Then, we linearize along that branch
and we find the parameters $\lambda_{c}$ at which
$L(\lambda_{c})$ has a non-trivial kernel.  At such a value we must
verify that $L(\lambda_{c})$ is Fredholm of index zero with a
one-dimensional kernel.  The space then splits along the kernel and the
range, the range component is eliminated by the analytic implicit
function theorem, and the problem reduces to a scalar equation. 
Solving this equation requires the transversality condition, that is, the
vanishing eigenvalue must cross zero with non-zero speed.

More in details, Section \ref{sec:functional} sets up the functional framework. We
compute the exact Fourier symbol of the derivative-containing part of
$\mathcal{C}_{\sigma}$, which turns out to be $\min\{i,j,\ell,m\}$ over the
resonant set $i-j+\ell=m$, and deduce that
$\mathcal F:X^{s}\times(0,\infty)\to X^{s}$ is real analytic for $s>3/2$. We also record
the embedding $X^{s}\hookrightarrow W^{1/4,4}(\mathbb{T})$, so that all the
profiles constructed here belong to the energy space used in
\cite{ConstantinDritschelGermain2025}, while solving \eqref{eq:TW} in the
stronger $H^{s}$ sense.
 
Section \ref{sec:spectrum} computes the linearization
$L(\lambda)=D_{\Phi}\mathcal F(\Phi_{0}(\lambda),\lambda)$. It couples the frequency $r$
only with its reflection $2k-r$, and $X^{s}$ splits accordingly into
\begin{itemize}
  \item $k-1$ blocks of size $2\times 2$, on
        $\operatorname{span}_{\mathbb{R}}\{e^{iqx},e^{i(2k-q)x}\}$ with
        $1\leq q<k$;
  \item one scalar block at $r=k$, on which $L(\lambda)$ acts as
        multiplication by $-2\lambda(k-\sigma)$, hence never singular for
        $\lambda>0$ and $k>\sigma$;
  \item a diagonal high-frequency tail on $r\geq 2k$, with symbol
        $\mathcal D_{r}(\lambda)=\omega\bigl(\tfrac{1}{r}-\tfrac{1}{k}\bigr)
        -\lambda(k-\sigma)$.
\end{itemize}
On the tail the unbounded contribution $\lambda\Lambda h$ cancels
against the mixed non-local term. This is the linearized counterpart of the null
structure of Section \ref{sec:functional}.
 
Section \ref{sec:thresholds} determines the critical values. The
determinant of the $n$-th block is a quadratic polynomial in
$\lambda$ whose discriminant is governed by the integer
$Q(k,n,\sigma)=(k+n-\sigma)^{2}-n(2k-n)$. It has two positive simple roots
$\lambda^{(n)}_{\pm}$ when $\omega<0$, and none when $\omega\geq 0$. 
 When \(\omega<0\), the tail contributes a second, disjoint family of
critical values
\(\lambda_r^{\mathrm{tail}}\), $r\geq2k$, which increase to
\[
  \lambda_\infty=-\frac{\omega}{k(k-\sigma)}.
\]
At \(\lambda_\infty\) no individual tail block vanishes, since
\[
  \mathcal D_r(\lambda_\infty)=\frac{\omega}{r},
\]
but these symbols converge to zero as \(r\to\infty\). Hence
\(L(\lambda_\infty)\) is not Fredholm. Moreover,
\[
  \ker_{X^s}L(\lambda_\infty)
  =
  \begin{cases}
    \operatorname{span}_{\mathbb R}\{e^{ix}\},
      & (k,\sigma)=(2,1),\\
    \{0\}, & \text{otherwise}.
  \end{cases}
\]
See Remark~\ref{rmk:spherical-n1}.  A
uniform spectral gap on the tail must therefore be established before
applying Lyapunov-Schmidt reduction.
 We prove that this gap
holds automatically at every finite-block critical value, with the single
exception $(k,n,\sigma)=(2,1,1)$. On the other
hand, to guarantee global simplicity of the kernel we have an additional hypothesis. In fact, for $\sigma=0$ and $k=7$ the two
blocks $n=4$ and $n=6$ become singular at the same value $\lambda=-\omega/14$.
We give a sufficient condition which rules this out, namely that
$Q(k,n,\sigma)$ be not a perfect square.
 
Section \ref{sec:bifurcation} contains the main result. At a non-resonant
finite-block critical value the kernel is one-dimensional, spanned by a vector
supported on the two frequencies $n$ and $2k-n$, and a local real-analytic
branch of traveling waves bifurcates.
We also show that the leading correction to the bifurcation parameter vanishes,
$\lambda'(0)=0$.
 
Section \ref{sec:additional} treats the tail values
$\lambda^{\mathrm{tail}}_{r}$, which are always simple and give a separate
family of scalar bifurcations. 

In particular, when \(\omega<0\), for \(\sigma=0\) and $k=1$, the $2\times 2$ blocks
are absent, and the finite tail thresholds are the only values
\(\lambda>0\) at which \(L(\lambda)\) has a non-trivial kernel.
Their limit \(\lambda_\infty\) is nevertheless a
bifurcation point in the topological sense. See
Remark~\ref{rmk:accumulation}.

The results below concern the local bifurcation of traveling wave
profiles.  The critical values are those at which the linearization of
the profile equation \eqref{eq:TW} loses invertibility.  This condition
does not by itself determine the spectral or nonlinear stability of
the corresponding traveling waves for the evolution
\eqref{eq:filamentation}.  
Our results are therefore complementary to the available stability
theory, including the orbital stability of $\Psi_1$ for $\sigma=0$
and of $\Psi_1$ and $\Psi_2$ for $\sigma=1$
\cite{ConstantinDritschelGermain2025}, as well as the linear and
modulational stability of harmonic waves established by Biello and
Hunter \cite{BielloHunter2010} within a semiclassical approximation of
the $\sigma=0$ equation.
For the related Burgers-Hilbert equation,
spectral instability of a traveling wave has recently been established
by a computer-assisted argument
\cite{CastroGomezSerranoPascualCaballo2026}.

\subsection{Notation}
\label{sub:notation}
 
Throughout, $\mathbb{T}=\mathbb{R}/(2\pi\mathbb{Z})$, $\mathbb P_{+}$ is the projector
onto the frequencies $m\geq 1$, $\Lambda=|\partial_{x}|$ and $\Lambda^{-1}$ is
its inverse on the mean-free functions. We write
$\widehat{u}_{m}$ or $a_{m}$ for the $m$-th Fourier coefficient of $u$. The
parameters $\omega\in\mathbb{R}$, $\sigma\in\{0,1\}$ and the wavenumber
$k\geq 1$ with $k>\sigma$ are fixed once and for all; $\lambda=A^{2}>0$ is the
bifurcation parameter and $c=c(\lambda)$ is given by \eqref{eq:dispersion}. The
case $k=\sigma=1$ is excluded. There the main block of the linearization vanishes identically and the bifurcation theory developed below does not
apply. Constants denoted $C_{s}$ depend only on $s$ and may change from line to line.

Unless otherwise stated, \(L(\lambda)\) is regarded as an operator
on $X^s$, and all its kernels are computed in $X^s$.
On $X^s$ we use the real \(L^2\) pairing
\begin{equation}\label{eq:real-pairing}
    \langle u,v\rangle_{\mathbb R}
    :=
    \operatorname{Re}
    \left(
        \frac1{2\pi}\int_{\mathbb T}u(x)\overline{v(x)}\,dx
    \right).
\end{equation}

\section{Functional setting}
\label{sec:functional}

For \(s\geq0\), let
\[
    H^s_+(\mathbb T)
    =
    \left\{
        u(x)=\sum_{m\geq1}a_m e^{imx}:
        \sum_{m\geq1}m^{2s}|a_m|^2<\infty
    \right\}.
\]
Define the closed real subspace
\begin{equation}\label{eq:Xs}
    X^s
    =
    \left\{
        u\in H^s_+(\mathbb T):
        \widehat u_m\in\mathbb R\ \text{for every }m\geq1
    \right\}.
\end{equation}
This is a closed real linear subspace of $H^s_+(\mathbb T)$, and it
will always be regarded as a real Hilbert space.  
The generator $i\Phi_0$ of the symmetries \eqref{eq_gen_symm} has the single Fourier
coefficient $iA$ at the frequency $k$, which is purely imaginary
because $A>0$. Hence $i\Phi_0\notin X^s$, and the null direction
common to the two symmetries is removed by the restriction to $X^s$.
The map $\mathcal F$ defined in \eqref{eq:F} preserves $X^s$.

\begin{rmk}[Relation with the variational energy space]
Constantin, Dritschel, and Germain
\cite[Proposition~5.3 and Section~6.3]{ConstantinDritschelGermain2025}
construct constrained minimizers of the energy
$E_\sigma$ in the positive-frequency space
$W^{1/4,4}(\mathbb T)$. Their Euler-Lagrange equation has the form
\[
    \frac{4}{\pi}\mathcal C_\sigma[u]
    =
    a\,\Lambda^{-1}u+b\,u,
\]
and is therefore equivalent, after relabeling the Lagrange multipliers,
to the traveling wave equation~\eqref{eq:TW}.

On the one dimensional torus, the Rellich-Kondrachov theorem gives
\[
    H^s(\mathbb T)
    \hookrightarrow\hookrightarrow
    W^{1/4,4}(\mathbb T)
    \qquad\text{for every }s>\frac12.
\]
Therefore, in the regime $s>3/2$ considered here,
\[
    X^s
    \hookrightarrow\hookrightarrow
    W^{1/4,4}(\mathbb T).
\]
Thus all the local bifurcating profiles constructed in this paper
belong to the variational energy space used in
\cite{ConstantinDritschelGermain2025}, while satisfying the
traveling wave equation in the stronger $H^s$ sense.
Notice that the two approaches are complementary. The variational theory provides
global constrained minimizers in the weak energy space, whereas the present analysis constructs and classifies branches bifurcating from single Fourier mode waves.
\end{rmk}

Recalling the definition of $\mathcal C_\sigma$ in \eqref{eq:cubic-Lambda} we prove now that the difference between $|u|^{2}\Lambda u$ and $u\Lambda|u|^{2}$ has a null structure which removes the apparent loss of one derivative.

\begin{lem}\label{lem:null}
Let
\[
    G[u]
    :=
    \mathbb P_+ \left[
        |u|^2\Lambda u-u\Lambda|u|^2
    \right].
\]
If \(u=\sum_{r\geq1}a_r e^{irx}\), then for every \(m\geq1\),
\begin{equation}\label{eq:null-symbol}
    \widehat{G[u]}_m
    =
    \sum_{\substack{i,j,\ell\geq1\\i-j+\ell=m}}
    \min\{i,j,\ell,m\}\,
    a_i\overline{a_j}a_\ell.
\end{equation}
Therefore, for every \(s>3/2\),
\begin{equation}\label{eq:null-Hs}
    \|G[u]\|_{H^s}
    \leq C_s\|u\|_{H^s}^3.
\end{equation}
Moreover,
\begin{equation}\label{eq:null-Lipschitz}
    \|G[u]-G[v]\|_{H^s}
    \leq
    C_s\big(\|u\|_{H^s}^2+\|v\|_{H^s}^2\big)
    \|u-v\|_{H^s}.
\end{equation}
\end{lem}

\begin{proof}
The coefficient at frequency \(m>0\) is
\[
    \widehat{G[u]}_m
    =
    \sum_{\substack{i,j,\ell\geq1\\i-j+\ell=m}}
    \big(\ell-|i-j|\big)
    a_i\overline{a_j}a_\ell.
\]
The summation set and the product
\(a_i\overline{a_j}a_\ell\) are invariant under
\(i\leftrightarrow\ell\).  Symmetrizing gives the symbol
\[
    \frac{i+\ell-|i-j|-|\ell-j|}{2}.
\]
Since \(i+\ell=j+m\) and all four indices are positive, a direct
case distinction yields
\[
    \frac{i+\ell-|i-j|-|\ell-j|}{2}
    =
    \min\{i,j,\ell,m\},
\]
which proves \eqref{eq:null-symbol}.

For the estimate, \(m=i-j+\ell>0\) implies \(m\leq i+\ell\), hence
\[
    m^s\leq C_s(i^s+\ell^s).
\]
Moreover \(\min\{i,j,\ell,m\}\leq j\).  If
\[
    A_r=r^s|a_r|,
    \qquad
    B_r=r|a_r|,
    \qquad
    C_r=|a_r|,
\]
then \(A\in\ell^2\), while \(B,C\in\ell^1\) for \(s>3/2\):
\[
    \|B\|_{\ell^1}
    \leq
    \left(\sum_{r\geq1}r^{2s}|a_r|^2\right)^{1/2}
    \left(\sum_{r\geq1}r^{2-2s}\right)^{1/2},
\]
and the corresponding estimate for \(C\) only requires \(s>1/2\).

Using the definitions of $A$, $B$, and $C$, we obtain, for every
$m\geq1$,
\begin{align*}
  m^s\bigl|\widehat{G[u]}_m\bigr|
  &\leq
  C_s
  \sum_{\substack{i,j,\ell\geq1\\i-j+\ell=m}}
  (i^s+\ell^s)j\,|a_i|\,|a_j|\,|a_\ell| \\
  &=
  C_s
  \sum_{\substack{i,j,\ell\geq1\\i-j+\ell=m}}
  \bigl(A_iB_jC_\ell+C_iB_jA_\ell\bigr).
\end{align*}
Extend $A$, $B$, and $C$ by zero to $\mathbb Z$, and define the
reflected sequence
\[
  \widetilde B_n:=B_{-n},
  \qquad n\in\mathbb Z.
\]
Then
\[
  \sum_{\substack{i,j,\ell\geq1\\i-j+\ell=m}}
  A_iB_jC_\ell
  =
  (A*\widetilde B*C)_m,
\]
and similarly
\[
  \sum_{\substack{i,j,\ell\geq1\\i-j+\ell=m}}
  C_iB_jA_\ell
  =
  (C*\widetilde B*A)_m.
\]
The two convolutions are equivalent by
the symmetry $i\leftrightarrow\ell$. Therefore, Young's inequality gives
\begin{align*}
  \|G[u]\|_{H^s}
  &\leq
  C_s\|A*\widetilde B*C\|_{\ell^2(\mathbb Z)} \\
  &\leq
  C_s\|A\|_{\ell^2}
       \|\widetilde B\|_{\ell^1}
       \|C\|_{\ell^1} \\
  &=
  C_s\|A\|_{\ell^2}
       \|B\|_{\ell^1}
       \|C\|_{\ell^1}
  \leq
  C_s\|u\|_{H^s}^3,
\end{align*}
which proves \eqref{eq:null-Hs}.

For the Lipschitz estimate, for $f,g,h\in H^s_+$, we define
\[
  \widehat{\mathcal T(f,g,h)}_m
  :=
  \sum_{\substack{i,j,\ell\geq1\\i-j+\ell=m}}
  \min\{i,j,\ell,m\}\,
  \widehat f_i\,\overline{\widehat g_j}\,\widehat h_\ell,
  \qquad m\geq1.
\]
This map is trilinear over $\mathbb R$ and satisfies
\[
  G[u]=\mathcal T(u,u,u).
\]
The preceding convolution argument, applied to three distinct
arguments, gives
\begin{equation}\label{eq:bound_tril}
    \|\mathcal T(f,g,h)\|_{H^s}
  \leq
  C_s
  \|f\|_{H^s}\|g\|_{H^s}\|h\|_{H^s}.
\end{equation}
Moreover,
\begin{align*}
  G[u]-G[v]
  =
  \mathcal T(u-v,u,u)
  +\mathcal T(v,u-v,u) +\mathcal T(v,v,u-v).
\end{align*}
Consequently,
\begin{align*}
  \|G[u]-G[v]\|_{H^s}
  &\leq
  C_s
  \bigl(
    \|u\|_{H^s}^2
    +\|u\|_{H^s}\|v\|_{H^s}
    +\|v\|_{H^s}^2
  \bigr)
  \|u-v\|_{H^s} \\
  &\leq
  C_s
  \bigl(
    \|u\|_{H^s}^2+\|v\|_{H^s}^2
  \bigr)
  \|u-v\|_{H^s},
\end{align*}
which proves \eqref{eq:null-Lipschitz}.
\end{proof}

\begin{prop}[Regularity of the traveling wave map]\label{prop:regularity}
Let \(s>3/2\).  Then
\begin{equation}\label{eq:F-Xs}
    \mathcal F:X^s\times(0,\infty)\longrightarrow X^s
\end{equation}
defined in \eqref{eq:F} is well defined and real analytic.
\end{prop}

\begin{proof}
Let $\mathcal T$ be the bounded real-trilinear map introduced in the
proof of Lemma~\ref{lem:null}, so that
\[
  G[u]=\mathcal T(u,u,u).
\]
Define also
\[
  \mathcal S(f,g,h)
  :=
  \mathbb P_+\bigl(f\,\overline g\,h\bigr).
\]
Since $H^s(\mathbb T)$ is a Banach algebra for $s>1/2$ and
$\mathbb P_+$ is bounded on $H^s(\mathbb T)$,
\[
  \|\mathcal S(f,g,h)\|_{H^s}
  \leq
  C_s\|f\|_{H^s}\|g\|_{H^s}\|h\|_{H^s}.
\]
Recalling \eqref{eq:bound_tril}, the map
\[
  \mathcal Q_\sigma
  :=
  \mathcal T-\sigma\mathcal S
\]
is bounded and real-trilinear on $(H^s_+)^3$, and
\[
  \mathcal C_\sigma[u]
  =
  \mathcal Q_\sigma(u,u,u).
\]
In particular,
\[
  \|\mathcal C_\sigma[u]\|_{H^s}
  \leq
  C_s\|u\|_{H^s}^3.
\]
We next verify that $\mathcal Q_\sigma$ preserves $X^s$.  Its Fourier
coefficients are
\[
  \widehat{\mathcal Q_\sigma(f,g,h)}_m
  =
  \sum_{\substack{i,j,\ell\geq1\\i-j+\ell=m}}
  \bigl(\min\{i,j,\ell,m\}-\sigma\bigr)
  \widehat f_i\,
  \overline{\widehat g_j}\,
  \widehat h_\ell .
\]
If $f,g,h\in X^s$, all their Fourier coefficients are real.  Since the
symbol $\min\{i,j,\ell,m\}-\sigma$ is also real for $\sigma=0,1$, every coefficient is real.  Hence
\[
  \mathcal Q_\sigma(X^s,X^s,X^s)\subseteq X^s.
\]
Moreover, $\widehat{\Lambda^{-1}u}_m = \frac{\widehat u_m}{m}$, and satisfies $\|\Lambda^{-1}u\|_{H^{s}}\leq\|u\|_{H^s}$.
It follows that, for every $u\in X^s$ and $\lambda>0$,
\[
  \mathcal F(u,\lambda)\in X^s
\]
and
\[
  \|\mathcal F(u,\lambda)\|_{H^s}
  \leq
  \left(
    \frac{|\omega|}{k}
    +\lambda|k-\sigma|
    +|\omega|
  \right)\|u\|_{H^s}
  +C_s\|u\|_{H^s}^3.
\]
Thus \eqref{eq:F-Xs} is well defined.
Finally, we write
\[
  \mathcal F(u,\lambda)
  =
  -\frac{\omega}{k}u
  +(k-\sigma)\lambda u
  +\omega\Lambda^{-1}u
  -\mathcal Q_\sigma(u,u,u).
\]
The first and third terms are bounded linear functions of $u$, the
second is a bounded bilinear function of $(\lambda,u)$, and the last is
a bounded cubic polynomial in $u$. Hence $\mathcal F$ is real analytic on
$X^s\times(0,\infty)$.
\end{proof}

\section{Linearized operator and its spectrum}
\label{sec:spectrum}
Recall that $\Phi_0(\lambda)=A e^{ikx}$, where $A=\sqrt{\lambda}$.
For each $\lambda>0$, we define the linearized operator along the trivial
branch by
\begin{equation}\label{eq:Llambda}
  L(\lambda)[h]
  :=
  D_\Phi\mathcal F(\Phi_0(\lambda),\lambda)[h]
  =
  -c(\lambda)h
  +\omega\Lambda^{-1}h
  -D\mathcal C_\sigma(\Phi_0(\lambda))[h],
  \qquad h\in X^s.
\end{equation}
The part of \eqref{eq:F} which does not come from $\mathcal C_\sigma$ acts
diagonally in the Fourier basis. In particular, on $e^{irx}$ it is multiplication by
\begin{equation}\label{eq:kinematic}
    \mathcal K_r(\lambda)
    :=
    -c(\lambda)+\frac{\omega}{r}
    =
    -\left(\frac{\omega}{k}-\lambda(k-\sigma)\right)+\frac{\omega}{r}
    =
    \omega\left(\frac1r-\frac1k\right)+\lambda(k-\sigma).
\end{equation}
Setting
\begin{equation}\label{eq:Omega}
    \Omega_r
    :=
    \omega\left(\frac1r-\frac1k\right),
\end{equation}
this reads $\mathcal K_r(\lambda)=\Omega_r+\lambda(k-\sigma)$. Notice that 
$\Omega_k=0$.  For each $1\leq q<k$ we let
\[
    p_q:=2k-q,
\]
so that
\begin{equation}\label{eq:Omega-explicit}
    \Omega_q=\omega\,\frac{k-q}{qk},
    \qquad
    \Omega_{p_q}=\omega\,\frac{q-k}{k(2k-q)}.
\end{equation}
Writing \(\lambda=A^2\), differentiation of \eqref{eq:cubic-Lambda} gives
\begin{multline}\label{eq:DC}
    D\mathcal C_\sigma(\Phi_0)[h]
    =
    \mathbb P_+\Big[
        \lambda\Lambda h
        +\lambda k\big(h+e^{2ikx}\overline h\big)\\
        -\lambda e^{ikx}
        \Lambda\big(e^{ikx}\overline h+e^{-ikx}h\big)
        -\sigma\lambda\big(2h+e^{2ikx}\overline h\big)
    \Big].
\end{multline}

\begin{lem}\label{lem:fourier-linear}
For $h=\sum_{r\geq1}h_r e^{irx}\in X^s$,
\begin{align}
    \widehat{D\mathcal C_\sigma(\Phi_0)[h]}_r
    &=
    \lambda\big(r+k-2\sigma-|r-k|\big)h_r
    +\mathbf 1_{\{1\leq r\leq2k-1\}}
    \lambda\big(k-\sigma-|r-k|\big)h_{2k-r}.
\label{eq:DC-Fourier}
\end{align}
Hence \(L(\lambda)\) couples \(r\) only with $2k-r$.
\end{lem}

\begin{proof}
Let us evaluate \eqref{eq:DC} on a single mode
$h=h_re^{irx}$, with $h_r\in\mathbb R$ and $r\geq1$ ($\overline h=h_re^{-irx}$). The four terms in \eqref{eq:DC} give, before the projection,
\begin{align*}
    \lambda\Lambda h
    &=
    \lambda r\,h_re^{irx},\\
    \lambda k\big(h+e^{2ikx}\overline h\big)
    &=
    \lambda k\,h_re^{irx}+\lambda k\,h_re^{i(2k-r)x},\\
    -\lambda e^{ikx}
    \Lambda\big(e^{ikx}\overline h+e^{-ikx}h\big)
    &=
    -\lambda|r-k|\,h_r\big(e^{i(2k-r)x}+e^{irx}\big),\\
    -\sigma\lambda\big(2h+e^{2ikx}\overline h\big)
    &=
    -2\sigma\lambda\,h_re^{irx}-\sigma\lambda\,h_re^{i(2k-r)x}.
\end{align*}
In the third line the argument of \(\Lambda\) is
\(h_r\big(e^{i(k-r)x}+e^{i(r-k)x}\big)\), and the two exponentials carry
the same weight \(|r-k|\).  Collecting the four lines, the bracket in
\eqref{eq:DC} equals
\begin{equation}\label{eq:mod_coup}
     \lambda\big(r+k-2\sigma-|r-k|\big)h_re^{irx}
    +\lambda\big(k-\sigma-|r-k|\big)h_re^{i(2k-r)x}.
\end{equation}
For a general $h$, linearity allows us to collect the contributions of
all its Fourier modes.  The first term above gives the diagonal
contribution
\[
  \lambda\bigl(r+k-2\sigma-|r-k|\bigr)h_r
\]
at the output frequency $r$.

A second contribution for the same frequency $r$ is obtained by considering the mode $2k-r$, provided that $2k-r\geq1$.  Indeed, replacing
$r$ by $2k-r$ in the second term in \eqref{eq:mod_coup} gives
\begin{align*}
  \lambda\bigl(k-\sigma-|(2k-r)-k|\bigr)
  h_{2k-r}e^{i(2k-(2k-r))x}=
  \lambda\bigl(k-\sigma-|r-k|\bigr)
  h_{2k-r}e^{irx}.
\end{align*}
No other mode can contribute to the frequency $r$, since
each mode produces only itself and its reflection about $k$.
Since $2k-r\geq1$ is equivalent to $1\leq r\leq2k-1$, collecting the
two contributions yields
\[
  \widehat{D\mathcal C_\sigma(\Phi_0)[h]}_r
  =
  \lambda\bigl(r+k-2\sigma-|r-k|\bigr)h_r
  +
  \mathbf 1_{\{1\leq r\leq2k-1\}}
  \lambda\bigl(k-\sigma-|r-k|\bigr)h_{2k-r}.
\]
This proves \eqref{eq:DC-Fourier}.
\end{proof}
By Lemma~\ref{lem:fourier-linear} and equation \eqref{eq:kinematic}, in the
ordered basis
\[
    \{e^{iqx},e^{ip_qx}\},
    \qquad 1\leq q<k,
\]
the restriction of \(L(\lambda)\) is the real symmetric matrix
\begin{equation}\label{eq:Mq}
    M^{(q)}(\lambda)
    =
    \begin{pmatrix}
        \Omega_q+\lambda(k-2q+\sigma)
        &
        -\lambda(q-\sigma)
        \\
        -\lambda(q-\sigma)
        &
        \Omega_{p_q}-\lambda(k-\sigma)
    \end{pmatrix}.
\end{equation}
Indeed \(q<k<p_q\) gives \(|q-k|=|p_q-k|=k-q\), whence
\begin{align*}
    M^{(q)}_{11}
    &=
    \Omega_q+\lambda(k-\sigma)
    -\lambda\big(q+k-2\sigma-(k-q)\big)\\
    &=\Omega_q+\lambda(k-2q+\sigma),\\
    M^{(q)}_{22}
    &=
    \Omega_{p_q}+\lambda(k-\sigma)
    -\lambda\big(p_q+k-2\sigma-(k-q)\big)\\
    &=\Omega_{p_q}-\lambda(k-\sigma),\\
    M^{(q)}_{12}=M^{(q)}_{21}
    &=
    -\lambda\big(k-\sigma-(k-q)\big)
    =-\lambda(q-\sigma),
\end{align*}
where the second line uses \(p_q+k-(k-q)=2k\).

The mode $r=k$ is a scalar block. Noting that $2k-r=k$, equation \eqref{eq:DC-Fourier} gives
\[
    D\mathcal C_\sigma(\Phi_0)[e^{ikx}]
    =
    \big[\lambda(2k-2\sigma)+\lambda(k-\sigma)\big]e^{ikx}
    =
    3\lambda(k-\sigma)e^{ikx}.
\]
Since \(\mathcal K_k(\lambda)=\lambda(k-\sigma)\),
\begin{equation}\label{eq:carrier-symbol}
    L(\lambda)[e^{ikx}]
    =
    \big(\lambda(k-\sigma)-3\lambda(k-\sigma)\big)e^{ikx}
    =
    -2\lambda(k-\sigma)e^{ikx}.
\end{equation}
For $r\geq2k$, the frequency $2k-r$ is non-positive and is
annihilated by $\mathbb P_+$.  Moreover the
contribution $\lambda\Lambda h$ cancels against the mixed non-local
term. Since $|r-k|=r-k$, the diagonal symbol of
\eqref{eq:DC-Fourier} is
\[
    \underbrace{\lambda r}_{\lambda\Lambda h}
    -\underbrace{\lambda(r-k)}_{\text{mixed term}}
    +\underbrace{\lambda k}_{\lambda kh}
    -\underbrace{2\sigma\lambda}_{2\sigma\lambda h}
    =
    2\lambda(k-\sigma).
\]
Thus the high-frequency ($r\geq2k$) tail is 
\begin{equation}\label{eq:tail}
    L(\lambda)[e^{irx}]
    =
    \big(\mathcal K_r(\lambda)-2\lambda(k-\sigma)\big)e^{irx}
    =
    \mathcal D_r(\lambda)e^{irx},
\end{equation}
with
\begin{equation}
    \mathcal D_r(\lambda)
    :=
    \Omega_r-\lambda(k-\sigma).
\end{equation}
Equivalently,
\begin{equation}\label{eq:tail-limit}
    \mathcal D_r(\lambda)
    =
    \mathcal D_\infty(\lambda)+\frac{\omega}{r},
    \qquad
    \mathcal D_\infty(\lambda)
    :=
    -\frac{\omega}{k}-\lambda(k-\sigma).
\end{equation}
Thus $X^s$ has the invariant block decomposition
\begin{equation}\label{eq:direct-sum}
    X^s
    =
    \bigoplus_{q=1}^{k-1}
    \operatorname{span}_{\mathbb R}
    \{e^{iqx},e^{i(2k-q)x}\}
    \oplus
    \operatorname{span}_{\mathbb R}\{e^{ikx}\}
    \oplus
    \overline{\operatorname{span}_{\mathbb R}
    \{e^{irx}:r\geq2k\}}^{\,X^s}.
\end{equation}
Ordering the basis accordingly as
\[
    \mathcal B
    =
    \big\{
         \mathcal B_1, \mathcal B_2,\dots, \mathcal B_{k-1},\
        e^{ikx},\
        e^{2ikx},\ e^{i(2k+1)x},\dots
    \big\},
    \qquad
     \mathcal B_q:=\{e^{iqx},e^{ip_qx}\},
\]
the operator \(L(\lambda)\) has the block-diagonal matrix
\begin{equation}\label{eq:block-matrix}
    \begin{pmatrix}
        \big[M^{(1)}\big] & \mathbf 0 & \cdots & \mathbf 0
            & \mathbf 0 & \mathbf 0 & \cdots\\
        \mathbf 0 & \big[M^{(2)}\big] & \cdots & \mathbf 0
            & \mathbf 0 & \mathbf 0 & \cdots\\
        \vdots & \vdots & \ddots & \vdots
            & \vdots & \vdots & \\
        \mathbf 0 & \mathbf 0 & \cdots & \big[M^{(k-1)}\big]
            & \mathbf 0 & \mathbf 0 & \cdots\\
        \mathbf 0 & \mathbf 0 & \cdots & \mathbf 0
            & -2\lambda(k-\sigma) & \mathbf 0 & \cdots\\
        \mathbf 0 & \mathbf 0 & \cdots & \mathbf 0
            & \mathbf 0 & \mathcal D_{2k} & \cdots\\
        \vdots & \vdots & & \vdots
            & \vdots & \vdots & \ddots
    \end{pmatrix},
\end{equation}
where each $M^{(q)}$, $1\leq q<k$, is the $2\times2$ symmetric
block \eqref{eq:Mq}, for $k$ we have the scalar
\eqref{eq:carrier-symbol}, and the tail consists of the scalars
$\mathcal D_r$ of \eqref{eq:tail}.  For $k=1$ the family of
$2\times2$ blocks is empty and the matrix is diagonal.

\begin{prop}
\label{prop:fredholm}
Fix \(\lambda_c>0\) and suppose that exactly one block
\(M^{(n)}(\lambda_c)\), \(1\leq n<k\), has a one-dimensional kernel.
Assume
\begin{equation}\label{eq:other-blocks}
    \det M^{(q)}(\lambda_c)\neq0
    \qquad
    \text{for every }1\leq q<k,\ q\neq n,
\end{equation}
and
\begin{equation}\label{eq:tail-gap}
    \gamma_c
    :=
    \inf_{r\geq2k}|\mathcal D_r(\lambda_c)|>0.
\end{equation}
If $k\neq\sigma$, then $L_c:=L(\lambda_c):X^s\to X^s$ is Fredholm of index zero, its kernel is one
dimensional, and, setting
\[
Y_c:=\bigl(\ker L_c\bigr)^{\perp_{\mathbb R}}\cap X^s ,
\]
one has $\operatorname{Ran}L_c=Y_c$ and $L_c|_{Y_c}:Y_c\to Y_c$ is invertible with bounded inverse.
\end{prop}

\begin{proof}
Assumption \eqref{eq:other-blocks} makes every $2\times2$ block except
the selected one invertible.  Since there are only finitely many such
blocks, the norms of their inverses have a finite maximum.  The carrier
block is also invertible, because its symbol is
\[
  -2\lambda_c(k-\sigma)\neq0
\]
by $\lambda_c>0$ and $k\neq\sigma$.

On the tail, assumption \eqref{eq:tail-gap} gives
\[
  \sup_{r\geq2k}
  \left|\mathcal D_r(\lambda_c)^{-1}\right|
  \leq
  \gamma_c^{-1}.
\]
Hence the reciprocal sequence defines a bounded Fourier multiplier on
the tail subspace of $X^s$.  Thus every block other than the selected
one is invertible with bounded inverse.

Let
\[
  K_c:=\ker L_c.
\]
Since the selected block is the only singular block and has a
one-dimensional kernel,
\[
  \dim K_c=1.
\]
Every matrix $M^{(q)}(\lambda_c)$ is real symmetric, while the carrier
block and the high-frequency tail are real and diagonal.  Therefore
$L_c$ is symmetric with respect to this pairing
\[
  \langle L_cf,g\rangle_{L^2,\mathbb R}
  =
  \langle f,L_cg\rangle_{L^2,\mathbb R},
  \qquad f,g\in X^s.
\]
We now prove that
\[
  \operatorname{Ran}L_c
  =
  K_c^{\perp_{L^2}}\cap X^s.
\]
Let $y=L_cx$ and let $v\in K_c$.  By symmetry,
\begin{align*}
  \langle y,v\rangle_{L^2,\mathbb R}
  &=
  \langle L_cx,v\rangle_{L^2,\mathbb R}=
  \langle x,L_cv\rangle_{L^2,\mathbb R}
  =0.
\end{align*}
Hence
\[
  \operatorname{Ran}L_c
  \subseteq
  K_c^{\perp_{L^2}}\cap X^s.
\]
Conversely, let
\[
  y\in K_c^{\perp_{L^2}}\cap X^s.
\]
On every block other than the selected one, the corresponding component
of $y$ has a preimage by the bounded invertibility established above.
In particular, on the tail this preimage is given by
\[
  x_r
  :=
  \frac{y_r}{\mathcal D_r(\lambda_c)},
  \qquad r\geq2k,
\]
and belongs to $X^s$ by \eqref{eq:tail-gap}.

On the selected block, the real symmetric matrix
$M^{(n)}(\lambda_c)$ satisfies
\[
  \operatorname{Ran}M^{(n)}(\lambda_c)
  =
  \bigl(\ker M^{(n)}(\lambda_c)\bigr)^\perp.
\]
Since the component of $y$ in the
selected block is orthogonal to $K_c$, it also has a preimage.
Combining the preimages on all the blocks gives an element $x\in X^s$
such that
\[
  L_cx=y.
\]
Therefore
\[
  K_c^{\perp_{L^2}}\cap X^s
  \subseteq
  \operatorname{Ran}L_c,
\]
and consequently
\[
  \operatorname{Ran}L_c
  =
  Y_c
  =
  K_c^{\perp_{L^2}}\cap X^s.
\]
Let $v^*$ generate $K_c$.  Since $v^*$ belongs to the selected
two-dimensional block, it has finite Fourier support.  The
$L^2$-orthogonal projection onto $K_c$ is
\[
  P_ch
  :=
  \frac{\langle h,v^*\rangle_{L^2,\mathbb R}}
       {\|v^*\|_{L^2}^2}\,v^*.
\]
This projection is bounded on $X^s$.  Indeed,
\[
  \|P_ch\|_{X^s}
  \leq
  \frac{\|v^*\|_{X^s}}{\|v^*\|_{L^2}}
  \|h\|_{L^2}
  \leq
  C\|h\|_{X^s}.
\]
Moreover,
\[
  \operatorname{Ran}P_c=K_c,
  \qquad
  \ker P_c=Y_c.
\]
Thus $Y_c$ is a closed subspace of $X^s$.  Since
$\operatorname{Ran}L_c=Y_c$, the range of $L_c$ is closed.

The decomposition
\[
  X^s=K_c\oplus Y_c
\]
shows that
\[
  \operatorname{codim}\operatorname{Ran}L_c
  =
  \operatorname{codim}Y_c
  =
  \dim K_c
  =
  1.
\]
Equivalently,
\[
  \dim\operatorname{coker}L_c=1.
\]
It follows that $L_c$ is Fredholm of index zero.

It remains to consider the restriction to $Y_c$.  It is injective
because
\[
  Y_c\cap\ker L_c=\{0\}.
\]
It is also surjective onto $Y_c$.  Indeed, if $y\in Y_c$, then
$y=L_cx$ for some $x\in X^s$, and
\[
  x-P_cx\in Y_c,
  \qquad
  L_c(x-P_cx)=L_cx=y.
\]
Hence
\[
  L_c|_{Y_c}:Y_c\longrightarrow Y_c
\]
is a bounded bijection.  Since $Y_c$ is a closed subspace of the Banach
space $X^s$, it is itself a Banach space.  The bounded inverse theorem
therefore implies that
\[
  \bigl(L_c|_{Y_c}\bigr)^{-1}:Y_c\longrightarrow Y_c
\]
is bounded.
\end{proof}

\begin{rmk}\label{rmk:pointwise-gap}
The pointwise condition
\(\mathcal D_r(\lambda_c)\neq0\) for every finite \(r\) does not by
itself imply \eqref{eq:tail-gap}.  If
\(\mathcal D_\infty(\lambda_c)=0\) and \(\omega\neq0\), then
\(\mathcal D_r(\lambda_c)=\omega/r\) is non-zero for every \(r\) but
tends to zero.  In that case \(0\) belongs to the spectrum and the range
of \(L_c\) is not closed.
\end{rmk}

\section{Critical values}
\label{sec:thresholds}

Fix \(1\leq n<k\) and set
\[
    p:=2k-n.
\]
Expanding \eqref{eq:Mq} with \(q=n\),
\begin{multline}\label{eq:det-expanded}
    \det M^{(n)}(\lambda)
    =
    \lambda^2\big[(k-2n+\sigma)(\sigma-k)-(n-\sigma)^2\big]\\
    +\lambda\big[\Omega_n(\sigma-k)+\Omega_p(k-2n+\sigma)\big]
    +\Omega_n\Omega_p.
\end{multline}
The three coefficients are evaluated as follows.  The leading one is
\[
    (k-2n+\sigma)(\sigma-k)-(n-\sigma)^2
    =
    -\big(k^2-2nk+2n\sigma-\sigma^2\big)
    -\big(n^2-2n\sigma+\sigma^2\big)
    =
    -(k-n)^2 .
\]
For the linear one, \eqref{eq:Omega-explicit} gives
\(\Omega_n=\omega(k-n)/(nk)\) and
\(\Omega_p=-\omega(k-n)/\big(k(2k-n)\big)\), whence
\begin{align*}
    \Omega_n(\sigma-k)+\Omega_p(k-2n+\sigma)
    &=
    -\frac{\omega(k-n)}{k}\cdot
    \frac{(k-\sigma)(2k-n)+n(k-2n+\sigma)}{n(2k-n)}\\
    &=
    -\frac{2\omega(k-n)^2(k+n-\sigma)}{kn(2k-n)}.
\end{align*}
The constant one is
\[
    \Omega_n\Omega_p
    =
    -\frac{\omega^2(k-n)^2}{k^2n(2k-n)} .
\]
Since $1\leq n<k$, one has $k-n\neq0$. Factoring out
$-(k-n)^2$, equation \eqref{eq:det-expanded} becomes

\begin{align}
    \det M^{(n)}(\lambda)
    &=
    -(k-n)^2 P_n(\lambda),
    \label{eq:det-factor}
\end{align}
where
\begin{align}
    P_n(\lambda)
    &:=
    \lambda^2
    -\frac{2\omega(\sigma-k-n)}{kn(2k-n)}\lambda
    +\frac{\omega^2}{k^2n(2k-n)}.
    \label{eq:Pn}
\end{align}
Define
\begin{equation}\label{eq:Q}
    Q(k,n,\sigma)
    :=
    (k+n-\sigma)^2-n(2k-n).
\end{equation}
Then,
\begin{equation}\label{eq:Q-cases}
    Q(k,n,0)=k^2+2n^2,
    \qquad
    Q(k,n,1)=(k-1)^2+2n(n-1).
\end{equation}
Thus \(Q(k,n,\sigma)>0\) whenever \(k>n\geq1\).  The discriminant of
\eqref{eq:Pn} is
\begin{equation}\label{eq:discriminant}
    \Delta
    =
    \frac{4\omega^2}{k^2n(2k-n)}
    \left[\frac{(\sigma-k-n)^2}{n(2k-n)}-1\right]
    =
    \frac{4\omega^2\,Q(k,n,\sigma)}{k^2n^2(2k-n)^2},
\end{equation}
so \(\Delta>0\) for every \(\omega\neq0\), and
\begin{equation}\label{eq:sqrt-discriminant}
    \sqrt\Delta
    =
    \frac{2|\omega|\sqrt{Q(k,n,\sigma)}}{kn(2k-n)} .
\end{equation}
If \(\omega<0\), then \(|\omega|=-\omega\) and
\(\omega(\sigma-k-n)=-\omega(k+n-\sigma)\), so the two roots of
\(P_n\) are
\begin{equation}\label{eq:lambdapm}
    \lambda_\pm^{(n)}
    =
    \frac{-\omega}{kn(2k-n)}
    \left[
        k+n-\sigma
        \pm\sqrt{Q(k,n,\sigma)}
    \right].
\end{equation}
Notice that since $k+n-\sigma>0$ and $(k+n-\sigma)^2-Q(k,n,\sigma)=n(2k-n)>0$,
we have
\[
  0<\sqrt{Q(k,n,\sigma)}<k+n-\sigma.
\]
In particular,
\[
  k+n-\sigma-\sqrt{Q(k,n,\sigma)}>0.
\]
This implies that the two roots in \eqref{eq:lambdapm} are both positive.

Conversely, for \(\omega\geq0\) there is no positive simple root of
\eqref{eq:Pn}.  Indeed, by \eqref{eq:Pn} the product of the roots is
\(\omega^2/\big(k^2n(2k-n)\big)\) and their sum is
\(2\omega(\sigma-k-n)/\big(kn(2k-n)\big)\).  Since \(0<n<k<2k\), the
product is strictly positive whenever \(\omega\neq0\), so the two roots
have the same sign. Moreover, since \(\sigma\in\{0,1\}\) and \(k>n\geq1\) force
\(\sigma-k-n<0\), the sum has the sign of \(-\omega\).  Hence for
\(\omega>0\) both roots are negative, while for \(\omega=0\) one has
\(P_n(\lambda)=\lambda^2\), whose only root is \(\lambda=0\) and is
double.  Since the roots in \eqref{eq:lambdapm} are distinct,
\begin{equation}\label{eq:det-transverse}
    \partial_\lambda\det M^{(n)}(\lambda_c)\neq0
\end{equation}
at either critical value \(\lambda_c=\lambda_\pm^{(n)}\).  Explicitly,
\(\partial_\lambda\det M^{(n)}(\lambda)=-(k-n)^2P_n'(\lambda)\) and
\(P_n'\big(\lambda_\pm^{(n)}\big)=\pm\sqrt\Delta\), so that by
\eqref{eq:sqrt-discriminant}
\begin{equation}\label{eq:det-transverse-explicit}
    \partial_\lambda\det M^{(n)}\big(\lambda_\pm^{(n)}\big)
    =
    \mp(k-n)^2\sqrt\Delta
    =
    \mp\,\frac{2(k-n)^2|\omega|\sqrt{Q(k,n,\sigma)}}{kn(2k-n)} .
\end{equation}
The critical values arising from
$M^{(q)}(\lambda)$, $1\leq q<k$, do not exhaust the parameter values
at which $L(\lambda)$ is singular.  Indeed, the high-frequency tail is
diagonal and satisfies
\[
  L(\lambda)[e^{irx}]
  =
  \mathcal D_r(\lambda)e^{irx},
  \qquad r\geq2k.
\]
Thus the mode $e^{irx}$ belongs to the kernel when
$\mathcal D_r(\lambda)=0$.  For $\omega<0$, $k>\sigma$, and every
integer $r\geq2k$, this occurs when 
\begin{equation}\label{eq:lambda-tail}
  \lambda_r^{\mathrm{tail}}
  :=
  -\omega\frac{r-k}{kr(k-\sigma)}>0.
\end{equation}
Equivalently,
\[
  \lambda_r^{\mathrm{tail}}
  =
  \lambda_\infty\left(1-\frac{k}{r}\right),
  \qquad
  \lambda_\infty
  :=
  -\frac{\omega}{k(k-\sigma)}.
\]
Consequently,
\begin{equation}\label{eq:tail-accumulation}
  \lambda_r^{\mathrm{tail}}
  \longrightarrow
  \lambda_\infty
  \qquad\text{as }r\to\infty.
\end{equation}
At the limiting value $\lambda_\infty$ one has
\[
  \mathcal D_r(\lambda_\infty)=\frac{\omega}{r}\longrightarrow0.
\]
Hence no finite tail block is singular at $\lambda_\infty$, but the tail gap closes and $L(\lambda_\infty)$ is not Fredholm.  The
simple Fredholm bifurcation theorem therefore does not apply there.

The relation between the critical values arising from the finite
coupled blocks, the finite tail thresholds
$\lambda_r^{\mathrm{tail}}$, and the accumulation point
$\lambda_\infty$ is examined below.
 
\begin{lem}[Tail gap]\label{lem:gap}
Let $\omega<0$, $\sigma\in\{0,1\}$, and let $k,n\in\mathbb Z$ with $k>n\ge1$ and $k>\sigma$. Then
\begin{equation}\label{eq:ordering}
\lambda^{(n)}_-\;<\;\lambda^{\mathrm{tail}}_{2k}\;\le\;\lambda^{\mathrm{tail}}_r\;<\;\lambda_\infty
\;\le\;\lambda^{(n)}_+
\qquad\text{for every }r\ge2k,
\end{equation}
and the last inequality is strict except for $(k,n,\sigma)=(2,1,1)$, where
$\lambda^{(1)}_+=\lambda_\infty$. Consequently
\begin{equation}\label{eq:gapauto}
\inf_{r\ge2k}\bigl|\mathcal D_r(\lambda_c)\bigr|>0
\end{equation}
for every $\lambda_c\in\{\lambda^{(n)}_-,\lambda^{(n)}_+\}$, with the single exception
$\lambda_c=\lambda^{(1)}_+$ in the case $(k,n,\sigma)=(2,1,1)$.
\end{lem}
 
\begin{proof}
The polynomial $P_n$ of \eqref{eq:Pn} has roots $\lambda^{(n)}_\pm$, so for real $x$
\begin{equation}\label{eq:signrule}
P_n(x)<0\iff \lambda^{(n)}_-<x<\lambda^{(n)}_+ .
\end{equation}
Substituting $\lambda^{\mathrm{tail}}_{2k}=\tfrac12\lambda_\infty$ and $\lambda_\infty$ into
\eqref{eq:Pn} and simplifying gives
\begin{align}
P_n\bigl(\lambda^{\mathrm{tail}}_{2k}\bigr)
&=-\,\frac{\omega^{2}\,(2k+n-4\sigma)}{4k^{2}(k-\sigma)^{2}(2k-n)},
\label{eq:P2k}\\[1mm]
P_n(\lambda_\infty)
&=-\,\frac{\omega^{2}\bigl[(k-\sigma)^{2}+n(n-2\sigma)\bigr]}{k^{2}n(2k-n)(k-\sigma)^{2}} .
\label{eq:Pinf}
\end{align}
Both denominators are positive, since $k>\sigma$ and $0<n<k<2k$.
 
In \eqref{eq:P2k} one has $2k+n-4\sigma>0$: for $\sigma=0$ this is immediate, and for $\sigma=1$
the constraint $k>\sigma$ forces $k\ge2$, whence $2k+n\ge5>4$. Therefore
$P_n(\lambda^{\mathrm{tail}}_{2k})<0$, and \eqref{eq:signrule} gives
$\lambda^{(n)}_-<\lambda^{\mathrm{tail}}_{2k}<\lambda^{(n)}_+$.
 
In \eqref{eq:Pinf} the bracket equals $k^{2}+n^{2}>0$ for $\sigma=0$, and $(k-1)^{2}+n(n-2)$ for
$\sigma=1$. In the latter case, if $n\ge2$ then $k>n$ forces $k\ge3$, so $(k-1)^{2}\ge4$ and
$n(n-2)\ge0$. If $n=1$ the bracket is $(k-1)^{2}-1$, positive for $k\ge3$ and zero exactly for
$k=2$. Hence $P_n(\lambda_\infty)<0$, giving $\lambda^{(n)}_-<\lambda_\infty<\lambda^{(n)}_+$,
except when $(k,n,\sigma)=(2,1,1)$, where $P_n(\lambda_\infty)=0$. In that case $\lambda_\infty$ is
a root of $P_n$, and since $\lambda_\infty>\lambda^{\mathrm{tail}}_{2k}>\lambda^{(n)}_-$ it is the
larger one, $\lambda_\infty=\lambda^{(1)}_+$. Together with the monotonicity of
$r\mapsto\lambda^{\mathrm{tail}}_r$ recorded in \eqref{eq:tail-accumulation}, this proves
\eqref{eq:ordering}.
 
Finally, $\mathcal D_r(\lambda_c)=0$ implies $\lambda_c=\lambda^{\mathrm{tail}}_r$ for some
$r\ge2k$, and $\mathcal D_\infty(\lambda_c)=0$ implies $\lambda_c=\lambda_\infty$. Both are
excluded by \eqref{eq:ordering} outside the stated exception. Since
$\mathcal D_r(\lambda_c)\to\mathcal D_\infty(\lambda_c)\neq0$, only finitely many indices satisfy
$|\mathcal D_r(\lambda_c)|<\tfrac12|\mathcal D_\infty(\lambda_c)|$, and each of them is non-zero,
hence \eqref{eq:gapauto}.
\end{proof}
 
\begin{lem}[Non-resonance criterion]\label{lem:nonres}
Let $\omega<0$, $\sigma\in\{0,1\}$, and let $k,n\in\mathbb Z$ with $k>n\ge1$ and $k>\sigma$. Let
$\lambda_c\in\{\lambda^{(n)}_-,\lambda^{(n)}_+\}$. If $Q(k,n,\sigma)$ is not a perfect square, then
\[
\det M^{(q)}(\lambda_c)\neq0
\qquad\text{for every }1\le q<k,\ q\neq n .
\]
\end{lem}
 
\begin{proof}
Since $\omega<0$ we may write $\lambda=|\omega|\mu=-\omega\mu$. By \eqref{eq:Pn},
$P_n(\lambda)=\omega^{2}\widetilde P_n(\mu)$ with
\[
\widetilde P_n(\mu)=\mu^{2}-\frac{2(k+n-\sigma)}{kn(2k-n)}\,\mu+\frac{1}{k^{2}n(2k-n)}
\ \in\ \mathbb Q[\mu],
\]
with roots $\mu^{(n)}_\pm=\bigl[(k+n-\sigma)\pm\sqrt{Q(k,n,\sigma)}\bigr]\big/\bigl(kn(2k-n)\bigr)$.
As $k,n,\sigma$ are integers, $Q(k,n,\sigma)$ is a positive integer. If it is not a perfect square
then $\sqrt{Q}$ is irrational, so $\mu_c:=\lambda_c/|\omega|$ is irrational and $\widetilde P_n$ is
its minimal polynomial over $\mathbb Q$. Suppose $\det M^{(q)}(\lambda_c)=0$ for some $q\neq n$. By
\eqref{eq:det-factor} this means $\widetilde P_q(\mu_c)=0$. Since $\widetilde P_q$ has
degree two with rational coefficients, $\widetilde P_q=\widetilde P_n$. Comparing constant terms
gives $q(2k-q)=n(2k-n)$, i.e. $(q-n)(2k-q-n)=0$. As $q\neq n$, this forces $q+n=2k$, impossible
because $q,n<k$.
\end{proof}
 
\begin{ex}\label{ex:k7}
Let $\sigma=0$, $\omega<0$ and $k=7$. Then $Q(7,4,0)=81=9^{2}$ and $Q(7,6,0)=121=11^{2}$ are both perfect
squares, and by \eqref{eq:lambdapm}
\[
\lambda^{(4)}_+=\frac{-\omega}{7\cdot4\cdot10}\,(11+9)=\frac{-\omega}{14},
\qquad
\lambda^{(6)}_+=\frac{-\omega}{7\cdot6\cdot8}\,(13+11)=\frac{-\omega}{14} .
\]
Writing $\lambda=-\omega\mu$ as in Lemma~\ref{lem:nonres}, a direct computation gives, for
$1\le q<7$,
\begin{equation}\label{eq:k7}
\widetilde P_q\Bigl(\tfrac1{14}\Bigr)=-\,\frac{(q-4)(q-6)}{196\,q\,(14-q)} ,
\end{equation}
which vanishes for $q=4$ and $q=6$. Hence at $\lambda_c=-\omega/14$ precisely two blocks are
singular and $\ker L(\lambda_c)$ is two dimensional, spanned by vectors supported on the
frequency pairs $\{4,10\}$ and $\{6,8\}$. Each of the two roots is simple within its own block by
\eqref{eq:det-transverse}. By Lemma~\ref{lem:gap},
$\lambda_c>\lambda_\infty=-\omega/49$, so the uniform tail gap does hold at $\lambda_c$ (the tail plays no role here).
\end{ex}

\section{Local bifurcation}
\label{sec:bifurcation}

\begin{defn}[Non-resonant critical value]\label{def:nonres}
    Let $\omega<0$, $\sigma\in\{0,1\}$ and $k,n\in\mathbb Z$ with $k>n\ge1$ and $k>\sigma$. A value
    $\lambda_c\in\{\lambda^{(n)}_-,\lambda^{(n)}_+\}$ is a \emph{non-resonant critical
    value} if
    \begin{equation}\label{eq:nonres}
    \det M^{(q)}(\lambda_c)\neq0\quad\text{for every }1\le q<k,\ q\neq n,
    \qquad\text{and}\qquad
    \lambda_c\neq\lambda_\infty .
    \end{equation}
    \end{defn}

    By Lemma~\ref{lem:gap} the second condition in \eqref{eq:nonres} fails only for
    $(k,n,\sigma)=(2,1,1)$ at $\lambda_c=\lambda^{(1)}_+$. In every other case that lemma also
    provides the uniform tail gap \eqref{eq:tail-gap} required by Proposition~\ref{prop:fredholm}.
    The first condition holds when
    $Q(k,n,\sigma)$ is not a perfect square by Lemma~\ref{lem:nonres}.

\begin{thm}[Local bifurcation]\label{thm:main}
Let $s>3/2$ and let $\lambda_c$ be a non-resonant critical value in the sense of Definition~\ref{def:nonres}.
Then
\[
    \ker L(\lambda_c)=\operatorname{span}_{\mathbb R}\{v^*\}
\]
for a normalized vector
\[
    v^*(x)
    =
    v_n^*e^{inx}+v_p^*e^{ipx},
    \qquad
    p=2k-n,
\]
and \((\Phi_0(\lambda_c),\lambda_c)\) is a local bifurcation point in
\(X^s\times(0,\infty)\).

More precisely, there exist \(\alpha_0>0\) and real-analytic maps
\[
    \alpha\longmapsto\lambda(\alpha),
    \qquad
    \alpha\longmapsto w(\alpha)\in X^s,
    \qquad |\alpha|<\alpha_0,
\]
such that
\begin{align}
    \lambda(0)&=\lambda_c,
    \qquad
    w(0)=0,
    \\
    \langle w(\alpha),v^*\rangle_{\mathbb R}&=0,
    \qquad
    w(\alpha)=O_{X^s}(\alpha^2),
\end{align}
and
\begin{equation}\label{eq:branch}
    \Phi(\alpha)
    :=
    \Phi_0(\lambda(\alpha))
    +\alpha v^*
    +w(\alpha)
\end{equation}
satisfies
\[
    \mathcal F(\Phi(\alpha),\lambda(\alpha))=0.
\]
The corresponding wave speed is
\[
    c(\alpha)
    =
    \frac{\omega}{k}-\lambda(\alpha)(k-\sigma).
\]
In a neighborhood of the bifurcation point, every solution in $X^s$
lies either on the trivial branch or on this curve.
\end{thm}

\begin{proof}
Define the shifted map
\begin{equation}\label{eq:shifted-map}
    \mathscr F(v,\lambda)
    :=
    \mathcal F(\Phi_0(\lambda)+v,\lambda).
\end{equation}
By Proposition~\ref{prop:regularity} and the analyticity of
$\lambda\mapsto\Phi_0(\lambda)$ on $(0,\infty)$, the map $\mathscr F:X^s\times(0,\infty)\longrightarrow X^s$ is real analytic near $(0,\lambda_c)$, and
$\mathscr F(0,\lambda)=0$.

At \(\lambda=\lambda_c\),
\[
    D_v\mathscr F(0,\lambda_c)=L(\lambda_c)=L_c.
\]
The selected determinant has a simple zero by
\eqref{eq:det-transverse}. Hence
\(M^{(n)}(\lambda_c)\) has rank one.  Hypothesis \eqref{eq:other-blocks} of Proposition~\ref{prop:fredholm} is the non-resonance condition assumed in Definition~\ref{def:nonres}, and the uniform tail gap \eqref{eq:tail-gap} is given by Lemma~\ref{lem:gap}. Then, Proposition~\ref{prop:fredholm} gives a one dimensional kernel and a bounded inverse on its complement.

Normalize $v^*$ by
\(\langle v^*,v^*\rangle_{\mathbb R}=1\), and define
\[
    Pu:=\langle u,v^*\rangle_{\mathbb R}v^*,
    \qquad
    Q:=I-P.
\]
Because $v^*$ has finite Fourier support, both $P$ and $Q$ are
bounded on $X^s$.
By Proposition~\ref{prop:fredholm},
\[
  \operatorname{Ran}L_c
  =
  (\ker L_c)^{\perp_{\mathbb R}}\cap X^s
  =
  QX^s, \qquad\text{and}\qquad
  L_c|_{QX^s}:QX^s\longrightarrow QX^s
\]
is invertible with bounded inverse.

Write
\[
    v=\alpha v^*+w,
    \qquad
    w\in QX^s.
\]
The auxiliary equation is
\begin{equation}\label{eq:aux-Hs}
    Q\mathscr F(\alpha v^*+w,\lambda)=0.
\end{equation}
Its derivative with respect to $w$ at
$(w,\alpha,\lambda)=(0,0,\lambda_c)$ is
\[
  D_w\left[Q\mathscr F(\alpha v^*+w,\lambda)\right]
  \big|_{(0,0,\lambda_c)}
  =
  L_c|_{QX^s}:QX^s\longrightarrow QX^s.
\]
Indeed, $Qh=h$ and $QL_ch=L_ch$ for every $h\in QX^s$.
Since $L_c|_{QX^s}$ is invertible with bounded inverse, the analytic
implicit function theorem gives a unique real-analytic map
\[
  w=w(\alpha,\lambda)\in QX^s
\]
near $(0,\lambda_c)$, with
\[
  w(0,\lambda)=0.
\]
Moreover, differentiating \eqref{eq:aux-Hs}
with respect to $\alpha$ at $(w,\alpha,\lambda)=(0,0,\lambda_c)$ gives
\[
  L_c|_{QX^s}\,\partial_\alpha w(0,\lambda_c)
  +
  QL_cv^*
  =
  0.
\]
Since $L_cv^*=0$ and the restriction of $L_c$ to $QX^s$ is
invertible, it follows that
\[
  \partial_\alpha w(0,\lambda_c)=0.
\]
It remains to solve
\[
  P\mathscr F\bigl(\alpha v^*+w(\alpha,\lambda),\lambda\bigr)=0.
\]
Since $Pu=\langle u,v^*\rangle_{\mathbb R}v^*$, this equation is
equivalent to 
\begin{equation}\label{eq:scalar-g}
  g(\alpha,\lambda)
  :=
  \left\langle
    \mathscr F\bigl(\alpha v^*+w(\alpha,\lambda),\lambda\bigr),
    v^*
  \right\rangle_{\mathbb R}
  =0.
\end{equation}
Since \(g(0,\lambda)=0\), there is a real-analytic function
\(\widetilde g\) such that
\[
    g(\alpha,\lambda)=\alpha\widetilde g(\alpha,\lambda).
\]
Moreover,
\begin{equation}\label{eq:der_g}
     \widetilde g(0,\lambda_c)
  =
  \partial_\alpha g(0,\lambda_c)
  =
  \langle L_cv^*,v^*\rangle_{\mathbb R}
  =0.
\end{equation}
Let \(\mu_1(\lambda)\) and \(\mu_2(\lambda)\) be the real-analytic eigenvalues of
the selected symmetric block $M^{(n)}$. Recalling that $0=\det M^{(n)}(\lambda_c)=
\mu_1(\lambda_c)\mu_2(\lambda_c)$ we fix 
\(\mu_1(\lambda_c)=0\) and
\[
    \mu_2(\lambda_c)\neq0.
\]
Then,
\[
    \partial_\lambda\det M^{(n)}(\lambda_c)
    =
    \mu_1'(\lambda_c)\mu_2(\lambda_c).
\]
By equation \eqref{eq:det-transverse}, we get \(\mu_1'(\lambda_c)\neq0\). 
As in \eqref{eq:der_g} and using $w(0,\lambda)=0$,
\[
\widetilde g(0,\lambda)=\bigl\langle L(\lambda)\bigl(v^*+\partial_\alpha w(0,\lambda)\bigr),v^*\bigr\rangle_{\mathbb R},
\]
whence
\[
\begin{aligned}
\partial_\lambda\widetilde g(0,\lambda_c)
&=\bigl\langle L'(\lambda_c)\bigl(v^*+\partial_\alpha w(0,\lambda_c)\bigr),v^*\bigr\rangle_{\mathbb R}
 +\bigl\langle L_c\,\partial^2_{\lambda\alpha}w(0,\lambda_c),v^*\bigr\rangle_{\mathbb R}\\
&=\bigl\langle L'(\lambda_c)v^*,v^*\bigr\rangle_{\mathbb R},
\end{aligned}
\]
because $\partial_\alpha w(0,\lambda_c)=0$, while the second vanishes
since $L_c$ is symmetric for \eqref{eq:real-pairing} and $L_cv^*=0$. 
Here $L'(\lambda_c):=\frac{d}{d\lambda}D_\Phi\mathcal F(\Phi_0(\lambda),\lambda)\big|_{\lambda=\lambda_c}$
denotes the total derivative along the branch, so that it includes $\Phi_0(\lambda)$. 

Finally $L'(\lambda_c)$
is represented by $\partial_\lambda M^{(n)}(\lambda_c)$ on
$\mathrm{span}_{\mathbb R}\{e^{inx},e^{ipx}\}$.
Let
\[
  \mathbf v^*:=(v_n^*,v_p^*)^T\in\mathbb R^2
\]
be the coordinate vector of $v^*$ in the basis
$\{e^{inx},e^{ipx}\}$. Let $\mathbf v(\lambda)\in\mathbb R^2$ be a
normalized eigenvector associated with the simple eigenvalue
$\mu_1(\lambda)$, chosen so that
\[
  \mathbf v(\lambda_c)=\mathbf v^*.
\]
Differentiating
\[
  M^{(n)}(\lambda)\mathbf v(\lambda)
  =
  \mu_1(\lambda)\mathbf v(\lambda)
\]
at $\lambda=\lambda_c$ and taking the scalar product with
$\mathbf v^*$ gives
\[
  \bigl\langle L'(\lambda_c)v^*,v^*\bigr\rangle_{\mathbb R}
  =
  \bigl\langle
    \partial_\lambda M^{(n)}(\lambda_c)\mathbf v^*,
    \mathbf v^*
  \bigr\rangle_{\mathbb R^2}
  =
  \mu_1'(\lambda_c)\neq0.
\]
Thus, by the above computation, we have 
\begin{equation}\label{eq:functional-transversality}
\partial_\lambda\widetilde g(0,\lambda_c)=\bigl\langle L'(\lambda_c)v^*,v^*\bigr\rangle_{\mathbb R}
=\mu_1'(\lambda_c)\neq0 .
\end{equation}
Therefore the implicit function theorem yields a unique
real-analytic function $\lambda=\lambda(\alpha)$ such that
\[
  \lambda(0)=\lambda_c,
  \qquad
  \widetilde g(\alpha,\lambda(\alpha))=0.
\]
Since $w(0,\lambda)=0$, we have $\partial_\lambda w(0,\lambda_c)=0$.
Together with $\partial_\alpha w(0,\lambda_c)=0$ and
$\lambda(0)=\lambda_c$, the chain rule gives
\begin{align*}
  \left.
  \frac{d}{d\alpha}
  w\bigl(\alpha,\lambda(\alpha)\bigr)
  \right|_{\alpha=0}
  &=
  \partial_\alpha w(0,\lambda_c)
  +
  \partial_\lambda w(0,\lambda_c)\lambda'(0)=0.
\end{align*}
Since $\alpha\mapsto w\bigl(\alpha,\lambda(\alpha)\bigr)=:w(\alpha)$ is real
analytic, it follows that $w(\alpha)=O_{X^s}(\alpha^2)$.

Finally, every nearby solution can be written uniquely as
$v=\alpha v^*+w$ with $w\in QX^s$. The auxiliary equation forces
$w=w(\alpha,\lambda)$, while the remaining equation is
$\alpha\widetilde g(\alpha,\lambda)=0$. Thus either $\alpha=0$ and the
solution belongs to the trivial branch, or $\alpha\neq0$ and
$\lambda=\lambda(\alpha)$, so that the solution belongs to the
bifurcating curve.
\end{proof}

\begin{cor}[Sufficient condition]\label{cor:square}
Let $s>3/2$, $\omega<0$, $\sigma\in\{0,1\}$, and let
$k,n\in\mathbb Z$ satisfy $k>n\ge1$ and $k>\sigma$.
If $Q(k,n,\sigma)$ is not a perfect square, then both
$\lambda^{(n)}_-$ and $\lambda^{(n)}_+$ are non-resonant critical values, and Theorem~\ref{thm:main}
applies at each of them.
\end{cor}
\begin{proof}
    The first condition in \eqref{eq:nonres} is Lemma~\ref{lem:nonres}. For the second, the only
    case with $\lambda_c=\lambda_\infty$ is $(k,n,\sigma)=(2,1,1)$ by Lemma~\ref{lem:gap}. There
    $Q(2,1,1)=(2+1-1)^{2}-1\cdot3=1=1^{2}$ is a perfect square, so that case is excluded by
    hypothesis.
\end{proof}

Theorem~\ref{thm:main} gives the existence and local uniqueness of the
bifurcating curve, but does not yet describe how the parameter
$\lambda$ varies along it.  Near $\alpha=0$, one has
\[
  \lambda(\alpha)
  =
  \lambda_c+\lambda'(0)\alpha+O(\alpha^2).
\]
Thus $\lambda'(0)$ measures the first order displacement of the
nontrivial branch from the critical value $\lambda_c$.  The next result shows that the linear correction vanishes, so the
first possible variation of $\lambda$ occurs at quadratic order.
Determining whether this quadratic correction is nonzero requires a
higher order analysis and is not addressed here.
\begin{prop}[Vanishing of the linear parameter correction]
\label{prop:lambda-prime}
Under the assumptions of Theorem~\ref{thm:main},
\begin{equation}\label{eq:lambda-prime}
    \lambda'(0)=0.
\end{equation}
\end{prop}

\begin{proof}
Let $d:=k-n>0$. Since $ n=k-d$ and $p=2k-n=k+d$
the kernel vector
$v^*=v_n^*e^{inx}+v_p^*e^{ipx}$ satisfies
\[
  \operatorname{supp}\widehat{v^*}
  \subseteq\{k-d,k+d\}.
\]
Every term in
\(D_\Phi^2\mathcal C_\sigma(\Phi_0(\lambda_c))[v^*,v^*]\)
contains one frequency $k$ and two frequencies in
\(\{k-d,k+d\}\).
Therefore,
\[
  \operatorname{supp}
  \widehat{
    D_\Phi^2\mathcal C_\sigma(\Phi_0(\lambda_c))[v^*,v^*]
  }
  \subseteq
  \{k-2d,k,k+2d\}\cap\mathbb N.
\]
These frequencies
are disjoint from \(\{k-d,k+d\}\), and hence
\begin{equation}\label{eq:quadratic-pairing}
    \left\langle
        D_\Phi^2\mathcal C_\sigma(\Phi_0(\lambda_c))[v^*,v^*],
        v^*
    \right\rangle_{\mathbb R}
    =0.
\end{equation}
Recalling \eqref{eq:branch}, set
\[
    v(\alpha)
    :=
    \Phi(\alpha)-\Phi_0(\lambda(\alpha))
    =\alpha v^*+w(\alpha).
\]
Differentiating twice at $\alpha=0$ the equation $\mathscr F\bigl(v(\alpha),\lambda(\alpha)\bigr)=0$,
where $\mathscr F$ is the shifted map defined in
\eqref{eq:shifted-map}, and pairing with $v^*$
gives
\[
    0
    =
    \left\langle L_c v''(0),v^*\right\rangle_{\mathbb R}
    +2\lambda'(0)
    \left\langle L'(\lambda_c)v^*,v^*\right\rangle_{\mathbb R}
    -
    \left\langle
        D_\Phi^2\mathcal C_\sigma(\Phi_0(\lambda_c))[v^*,v^*],
        v^*
    \right\rangle_{\mathbb R}.
\]
Here the other derivatives vanish because
\(\mathscr F(0,\lambda)=0\).  The first term also vanishes, since
\(L_c\) is symmetric and \(L_cv^*=0\), while the third vanishes by
\eqref{eq:quadratic-pairing}. This yields
\[
    2\lambda'(0)
    \left\langle L'(\lambda_c)v^*,v^*\right\rangle_{\mathbb R}
    =0.
\]
The pairing is nonzero by
\eqref{eq:functional-transversality}.  Thus \(\lambda'(0)=0\).
\end{proof}

\section{Tail bifurcations}
\label{sec:additional}
Recall from~\eqref{eq:lambda-tail} and~\eqref{eq:tail-accumulation} the tail thresholds
$\lambda^{\mathrm{tail}}_r$ and their limit $\lambda_\infty$. At $\lambda_\infty$ the tail gap
closes and the bifurcation theorem does not apply.

\begin{prop}[Tail bifurcations]\label{prop:tail-bif}
Let \(s>3/2\), assume \(\omega<0\) and \(k>\sigma\), and fix an
integer $r\geq2k$.
Then
\[
    \mathcal D_r(\lambda_r^{\mathrm{tail}})=0,
    \qquad
    \partial_\lambda\mathcal D_r
    =-(k-\sigma)\neq0.
\]
Moreover,
\((\Phi_0(\lambda_r^{\mathrm{tail}}),
\lambda_r^{\mathrm{tail}})\)
is a local bifurcation point in $X^s$, with kernel generated by
$e^{irx}$.  With the kernel coefficient as branch parameter,
\[
    \lambda(\alpha)
    =
    \lambda_r^{\mathrm{tail}}+O(\alpha^2).
\]
The corresponding speed is
\[
    c(\alpha)
    =
    \frac{\omega}{k}-\lambda(\alpha)(k-\sigma).
\]
\end{prop}

\begin{proof}
Fix $r\ge2k$ and write $\lambda_r:=\lambda^{\mathrm{tail}}_r$.
 
\emph{Gap on the complement of the mode $r$.} By \eqref{eq:tail} and \eqref{eq:lambda-tail},
$\mathcal D_j(\lambda_r)=\omega\bigl(\tfrac1j-\tfrac1r\bigr)\neq0$ for every $2k\leq j\neq r$,
and $\mathcal D_\infty(\lambda_r)=-\omega/r\neq0$ by \eqref{eq:tail-limit}. Since
$\mathcal D_j(\lambda_r)\to\mathcal D_\infty(\lambda_r)$, we have
\[ \inf_{\substack{j\ge2k\\j\neq r}} |\mathcal D_j(\lambda_r)| = |\omega| \inf_{\substack{j\ge2k\\j\neq r}} \left|\frac1j-\frac1r\right| = \frac{|\omega|}{r(r+1)}>0. \]
Moreover, the coefficient
\eqref{eq:carrier-symbol} for the mode $k$ is $-2\lambda_r(k-\sigma)\neq0$ by definition.
 
\emph{The $2\times2$ blocks are non-singular.} Applying Lemma~\ref{lem:gap} with $n=q$
gives, for every $1\le q<k$,
\[
\lambda^{(q)}_-<\lambda^{\mathrm{tail}}_{2k}\le\lambda_r<\lambda_\infty\le\lambda^{(q)}_+ ,
\]
so $\lambda_r$ lies between the two roots of $P_q$. Hence $P_q(\lambda_r)<0$ by
\eqref{eq:signrule}, and $\det M^{(q)}(\lambda_r)=-(k-q)^2P_q(\lambda_r)\neq0$ by
\eqref{eq:det-factor}. No non-resonance hypothesis is needed.
 
\emph{Kernel and Fredholm property.} Combining the previous paragraphs, every block of the
decomposition \eqref{eq:direct-sum} is invertible except the scalar at the index $r$,
whose symbol vanishes (the inverses are uniformly bounded). Therefore
$L(\lambda_r):X^s\to X^s$ is Fredholm of index zero with
$\ker L(\lambda_r)=\mathrm{span}_{\mathbb R}\{e^{irx}\}$, and its restriction to
$(\ker L(\lambda_r))^{\perp_{L^2}}\cap X^s$ is invertible with bounded inverse.
 
\emph{Transversality.} By \eqref{eq:tail}, $\partial_\lambda\mathcal D_r(\lambda)=-(k-\sigma)\neq0$,
so the argument of
\eqref{eq:functional-transversality} applies with $v^*=e^{irx}$.

\emph{Conclusion.} Lyapunov-Schmidt reduction as in the proof of Theorem~\ref{thm:main} yields a
real-analytic branch. 
The argument of Proposition~\ref{prop:lambda-prime} applies with
$v^*=e^{irx}$. Indeed, the quadratic term
\[
  D_\Phi^2\mathcal C_\sigma(\Phi_0(\lambda_r))
  [e^{irx},e^{irx}]
\]
is supported on the frequencies $\{k,2r-k\}$, which are disjoint from
$\{r\}$ because $r\geq2k>k$. Hence its pairing with $e^{irx}$
vanishes, and therefore $\lambda'(0)=0$.
\end{proof}

\begin{cor}\label{cor:k1}
Assume $\omega<0$. For $k=1$ and $\sigma=0$ there are no $2\times2$ blocks, and the coefficient for the mode $k$
\eqref{eq:carrier-symbol} equals $-2\lambda\neq0$. Hence the values
$\lambda^{\mathrm{tail}}_r=-\omega\,(r-1)/r$, $r\ge2$, are the values at which
$L(\lambda)$ has a non-trivial kernel, and each of them is a bifurcation value by
Proposition~\ref{prop:tail-bif}. Thus the $k=1$ branch has a bifurcation at
every \(\lambda_r^{\mathrm{tail}}\).
They accumulate at $\lambda_\infty=-\omega$, where $L(\lambda_\infty)$ has trivial kernel but is
not Fredholm, by Remark~\ref{rmk:pointwise-gap}.
\end{cor}

\begin{proof}
Notice that 
\[ \mathcal D_j(\lambda) = \omega\left(\frac1j-1\right)-\lambda, \]
hence, for \(\lambda>0\),
\[ \mathcal D_j(\lambda)=0 \quad\Longleftrightarrow\quad \lambda=-\omega\frac{j-1}{j}. \]
There are no finite $2\times2$ blocks and the symbol of mode $k$ is
\(-2\lambda\neq0\). At \(\lambda_\infty=-\omega\),
\(\mathcal D_j(\lambda_\infty)=\omega/j\), so the kernel is trivial but the operator is not Fredholm.
\end{proof}

\begin{rmk}[Bifurcation at the accumulation point]\label{rmk:accumulation}
Let $\omega<0$ and $k>\sigma$. 
Although no local branch is constructed directly from
$\lambda_\infty$, the accumulation of the tail bifurcation points makes
$(\Phi_0(\lambda_\infty),\lambda_\infty)$ a bifurcation point in a topological sense. Every neighborhood of this point contains
a non-trivial solution of $\mathcal F=0$.
Indeed, given $\varepsilon>0$, choose $r\ge2k$ so large that
\[
|\lambda_r^{\mathrm{tail}}-\lambda_\infty|
+
\|\Phi_0(\lambda_r^{\mathrm{tail}})
-\Phi_0(\lambda_\infty)\|_{X^s}
<\frac{\varepsilon}{2}.
\]
This is possible by \eqref{eq:tail-accumulation} and because
$\Phi_0(\lambda)=\sqrt{\lambda}\,e^{ikx}$.
Then choose a non-trivial point $(\Phi,\lambda)$ of the branch
of Proposition~\ref{prop:tail-bif} issuing from
$\lambda_r^{\mathrm{tail}}$ so close to its base point that
\[
|\lambda-\lambda_r^{\mathrm{tail}}|
+
\|\Phi-\Phi_0(\lambda_r^{\mathrm{tail}})\|_{X^s}
<\frac{\varepsilon}{2}.
\]
The triangle inequality gives the claim.
\end{rmk}

\begin{rmk}[The case $\sigma=1$, $n=1$]
\label{rmk:spherical-n1}
Assume $\omega<0$ and $k\ge2$.  Since the off-diagonal coefficient
$-\lambda(n-\sigma)$ vanishes when $n=\sigma=1$, one has
\[
  M^{(1)}(\lambda)
  =
  \operatorname{diag}\left(
    (k-1)\left(\frac{\omega}{k}+\lambda\right),
    (k-1)\left(-\frac{\omega}{k(2k-1)}-\lambda\right)
  \right).
\]
Consequently,
\[
  \lambda_-^{(1)}
  =
  -\frac{\omega}{k(2k-1)},
  \qquad
  \lambda_+^{(1)}
  =
  -\frac{\omega}{k},
\]
and the corresponding kernels are generated by
$e^{i(2k-1)x}$ and $e^{ix}$, respectively.

The upper critical value belongs to an explicit two-mode family.
Indeed, for every $a\in\mathbb R$,
\[
  \Phi_a(x)
  =
  a e^{ix}
  +
  \sqrt{-\frac{\omega}{k}}\,e^{ikx},
  \qquad
  \lambda(a)=-\frac{\omega}{k},
  \qquad
  c(a)=\omega,
\]
solves the traveling-wave equation.

Since
\[
  Q(k,1,1)=(k-1)^2
\]
is a perfect square, Corollary~\ref{cor:square} does not apply and
non-resonance must be checked directly. Writing
$\lambda=-\omega\mu$ as in Lemma~\ref{lem:nonres}, one obtains
\[
  \widetilde P_q\!\left(\frac1k\right)
  =
  \frac{-(q-1)(q-2k+3)}
       {k^2q(2k-q)},
\]
and
\[
  \widetilde P_q\!\left(\frac{1}{k(2k-1)}\right)
  =
  -\frac{(q-1)(q+2k-1)}
        {k^2q(2k-q)(2k-1)^2}.
\]
For $k\ge3$ and $1\le q<k$, each expression vanishes only at
$q=1$. The additional zero $q=2k-3$ of the first expression lies
outside the admissible range, while the additional zero $q=1-2k$ of
the second is negative. Hence
\[
  \det M^{(q)}(\lambda_\pm^{(1)})\neq0
  \qquad
  \text{for every }1\le q<k,\quad q\neq1.
\]
Moreover, Lemma~\ref{lem:gap} gives
\[
  \lambda_\pm^{(1)}\neq\lambda_\infty
  \qquad\text{for }k\ge3.
\]
Thus both critical values satisfy Definition~\ref{def:nonres} when
$k\ge3$.

At $\lambda_+^{(1)}$, Theorem~\ref{thm:main} identifies the local
bifurcating curve with the explicit family above. With the
normalization $v^*=e^{ix}$ and the kernel coordinate $\alpha=a$,
\[
  w(\alpha)=0,
  \qquad
  \lambda(\alpha)=-\frac{\omega}{k}.
\]
The branch is therefore a vertical branch, rather than a
non-degenerate pitchfork. The additional conclusion supplied by the
theorem is its local classification within $X^s$, not the existence
of the family itself. At $\lambda_-^{(1)}$, the theorem yields a local
real-analytic branch tangent to $e^{i(2k-1)x}$.

The case $k=2$ is exceptional. In this case,
\[
  \lambda_+^{(1)}
  =
  -\frac{\omega}{2}
  =
  \lambda_\infty,
  \qquad
  \lambda_-^{(1)}
  =
  -\frac{\omega}{6}.
\]
At $\lambda_+^{(1)}$, the tail symbols are
\[
  \mathcal D_j(\lambda_+^{(1)})
  =
  \frac{\omega}{j},
  \qquad j\ge4.
\]
They are non-zero for every finite $j$ but tend to zero as
$j\to\infty$. Thus the tail gap closes, the linearization is not
Fredholm, and Theorem~\ref{thm:main} does not apply.

Nevertheless, a vertical branch can be verified directly. Set
\[
  B:=\sqrt{-\frac{\omega}{2}}
\]
and, for every $a\in\mathbb R$, define
\[
  \Phi_a(x):=a e^{ix}+B e^{2ix}.
\]
Since $\lambda_+^{(1)}=-\omega/2$, the traveling-wave map becomes
\[
  \mathcal F(\Phi,\lambda_+^{(1)})
  =
  -\omega\Phi+\omega\Lambda^{-1}\Phi-\mathcal C_1[\Phi].
\]
Using
\[
  \Lambda^{-1}e^{ix}=e^{ix},
  \qquad
  \Lambda^{-1}e^{2ix}=\frac12e^{2ix},
\]
we obtain
\[
  -\omega\Phi_a+\omega\Lambda^{-1}\Phi_a
  =
  -\frac{\omega}{2}B e^{2ix}.
\]
On the other hand, the identity
\[
  \mathcal C_1[a e^{ix}+B e^{2ix}]
  =
  |B|^2B e^{2ix}
\]
and the relation $|B|^2=-\omega/2$ give
\[
  \mathcal C_1[\Phi_a]
  =
  -\frac{\omega}{2}B e^{2ix}.
\]
Therefore
\[
  \mathcal F(\Phi_a,\lambda_+^{(1)})=0.
\]
Hence
\[
  \Phi_a(x)
  =
  a e^{ix}
  +
  \sqrt{-\frac{\omega}{2}}\,e^{2ix},
  \qquad
  \lambda(a)=-\frac{\omega}{2},
  \qquad
  c(a)=\omega,
\]
is a vertical branch. Its existence follows directly from the
equation, independently of Theorem~\ref{thm:main}. 

On the other hand, $\lambda_-^{(1)}$ is non-resonant. Since $1\le q<2$, there are no
other finite $2\times2$ blocks and
$\lambda_-^{(1)}\neq\lambda_\infty$. Hence
Theorem~\ref{thm:main} gives a local real-analytic branch tangent to
$e^{3ix}$. This is the unique case in which one of the two roots
coincides with the non-Fredholm tail accumulation point while the
other remains a simple Fredholm bifurcation value.
\end{rmk}

\vspace{5mm}
\textbf{Acknowledgment.}
The author is grateful to Adrian Constantin and Pierre Germain for many helpful discussions during the preparation of this paper.

\vspace{1mm}
\textbf{Conflict of interest.} The author declares that he has no conflict of interest.

\vspace{1mm}
\textbf{Data availability.} Data sharing is not applicable.
We do not analyze or generate any datasets, because our work proceeds within a theoretical approach.

\bibliographystyle{plain}
\bibliography{bib}

\end{document}